\documentclass[12pt]{amsart} 
\usepackage{amssymb}
\usepackage{amsmath}
\usepackage{latexsym}
\usepackage{amscd}
\usepackage{amsfonts}
\usepackage{pb-diagram}
\usepackage[all]{xypic}
\usepackage[mathcal,mathscr]{eucal}
\usepackage{amsthm}
\usepackage{graphicx}
\usepackage{url}
\usepackage{pdfsync}
\usepackage{bbm}
\usepackage{xypic}
\usepackage[pdftex,colorlinks]{hyperref}
\usepackage{pifont}
\usepackage{amsbsy}
\usepackage{epstopdf}
\usepackage{times}
\usepackage{epic,eepic}

\newtheorem{theorem}{Theorem}[section]
\newtheorem{prop}[theorem]{Proposition}
\newtheorem{lemma}[theorem]{Lemma}
\newtheorem{corollary}[theorem]{Corollary}

\newtheorem{definition}[theorem]{Definition}

\theoremstyle{remark}
\newtheorem{remark}[theorem]{\bf {Remark}}
\newtheorem{example}[theorem]{\bf {Example}}

\numberwithin{equation}{section}

\DeclareMathOperator{\Sym}{\mathscr{S}ym}

\DeclareMathOperator{\Dom}{Dom}

\DeclareMathOperator{\Hom}{Hom}

\DeclareMathOperator{\Hess}{Hess}

\DeclareMathOperator{\cpt}{cpt}
\DeclareMathOperator{\Pf}{Pf}

\DeclareMathOperator{\End}{End}
\DeclareMathOperator{\Res}{Res}

\DeclareMathOperator{\Ker}{Ker}

\DeclareMathOperator{\Tr}{Tr}

\DeclareMathOperator{\Diffeo}{Diffeo}

\DeclareMathOperator{\supp}{supp}

\DeclareMathOperator{\ind}{ind}

\DeclareMathOperator{\even}{even}
\DeclareMathOperator{\grad}{grad}

\DeclareMathOperator{\tina}{\widetilde{\nabla}}

\DeclareMathOperator{\odd}{odd}
\DeclareMathOperator{\loc}{loc}
\DeclareMathOperator{\str}{str}

\DeclareMathOperator{\dist}{dist}

\DeclareMathOperator{\ch}{ch}
\DeclareMathOperator{\ad}{ad}

\DeclareMathOperator{\tr}{tr}
\DeclareMathOperator{\ct}{ct}
\DeclareMathOperator{\Bl}{{\bf B}l}

\DeclareMathOperator{\Vol}{Vol}
\DeclareMathOperator{\rank}{rank}

\DeclareMathOperator{\id}{id}

\DeclareMathOperator{\codim}{codim}

\DeclareMathOperator{\wstr}{w-str}

\DeclareMathOperator{\Lag}{\mathscr{L}ag}

\DeclareMathOperator{\Gr}{Gr}
\DeclareMathOperator{\Real}{Re}

\DeclareMathOperator{\Imag}{Im}

\newcommand{\bA}{\mathbb A}

\newcommand\bC{{\mathbb C}}

\newcommand{\bN}{{{\mathbb N}}}

\newcommand{\bP}{{{\mathbb P}}}

\newcommand\bR{{\mathbb R}}

\newcommand\bZ{{\mathbb Z}}

\newcommand{\eF}{\EuScript{F}}

\newcommand{\eP}{\EuScript{P}}

\newcommand{\eS}{\EuScript{S}}
\newcommand{\eT}{\EuScript{T}}
\newcommand{\eU}{\EuScript{U}}

\newcommand{\eSt}{\EuScript{S}}
\newcommand{\ul}{\underline}

\newcommand{\ra}{\rightarrow}

\newcommand{\ctr}{\mathrel{\llcorner}}

\begin{document}
\title{Vertical flows and a general currential homotopy formula}
\begin{abstract} We generalize  some of the results of Harvey, Lawson and Latschev about transgression formulas. The focus here is  on flowing forms via vertical vector fields, especially Morse-Bott-Smale vector fields.   We prove a very general  transgression formula including also a version covering non-compact situations. Among applications, we completely answer a question of Quillen,  construct the Maslov spark,  give a very short proof of a refined  Chern-Gauss-Bonnet theorem,   and reprove some theorems  of Nicolaescu and  Getzler. A discussion about odd Chern-Weil theory is also included.
\end{abstract}
\author{Daniel Cibotaru}
\address{Universidade Federal do Cear\'a, Fortaleza, CE, Brasil}
\email{daniel@mat.ufc.br}
\maketitle
\tableofcontents

\section{\bf {Introduction}}

One upshot of Harvey and Lawson's work on singular connections \cite{HL1} is a refinement of some classical results in differential topology and geometry. In essence, each of these classical theorems gives   two different representatives for the same cohomology class,  using different constructions for the cohomology theory of a manifold. For example, the celebrated Chern-Gauss-Bonnet in its general form asserts that on an oriented vector bundle with connection, the Euler form of the underlying connection determines a deRham cohomology class which is Poincar\'e dual to the zero locus of a generic section of the bundle. The refinement consists in giving a transgression formula that "quantifies" at the level of currents  how these two representatives differ.  Harvey and Lawson give  a recipe on how to get "canonically" a current whose differential is the difference of the two representatives. Moreover,  it turns out that the transgression formulas themselves are of a cohomological  nature. In fact,   using the "dichotomy" between smooth and integral currents, Harvey, Lawson and Zweck construct a new variant  of  \emph{smooth} cohomology theory \cite{HLZ} such that the groups of the so called de Rham-Federer differential characters  are isomorphic with the groups of differential characters of Cheeger and Simons \cite{CS}. The transgression equations naturally produce objects in these groups. 

In the past 20 years, the authors of  \cite{HL1} and their collaborators have discovered new and important links of their theory of transgression formulas  with other areas and in particular  with Morse theory  \cite{HL2,HL3,HL4,La,Na}. While \cite{HL1} was the starting point of the theory, the framework and results that motivated the writing of the present note are part of \cite{HL2} with major technical inputs from \cite{HL3} and \cite{La}. We should add that there is basically no mention of singular connections in this article.

Our initial project was to generalize some parts of \cite{HL2} to an infinite dimensional context and we obtained indeed some preliminary results in that direction \cite{Ci2}. However,  it turns out  that even in  finite dimensions one can give an expansion of the existing theory, embracing new examples and clarifying certain technical aspects.  This is what we tried to accomplish in the present note.   Roughly half of the article is made up of examples.  Another substantial part of it consists of new\footnote{to a lesser or greater  extent} proofs of known results using the theorems herein.   

Our main result is a very general homotopy formula.
\begin{theorem} \label{t11} Let $\pi:P\ra B$ be a fiber bundle of oriented manifolds.  Assume that $g$ is a Riemannian metric on $P$ with respect to which all fibers are complete, $\omega\in\Omega^k(P)$ is a smooth form, $f:P\ra \bR$ is a smooth, proper, fiberwise  Morse-Bott-Smale function satisfying a certain local-constancy condition and $\phi_0:B\ra P$ is a section transversal to the stable bundles of $\nabla^Vf$, the vertical gradient of $f$, assumed to be complete. Suppose furthermore that $(g,\omega,\nabla^V f, \phi_0)$ form a strongly atomic tuple (Def. \ref{sattup}) and that  $\omega$ satisfies a certain integrability condition  (see (\ref{eq2snv})). Then the following equality of locally flat currents holds:
\begin{equation}\label{t12} \lim_{t\ra \infty}\phi_t^*\omega = \sum_{\codim{S(F)}\leq k} \Res_F^u(\omega)[\phi_0^{-1}(S(F))],
\end{equation}
where $\phi_t$ is the flow of $\phi_0$ up to time $t$ via $\nabla^Vf$, $S(F)/U(F)$ are the stable/unstable bundles of $\nabla^V f$ and $\Res_F^u(\omega)=\tau_F^*\displaystyle\int_{U(F)/F}\omega$ is a certain residue form along $\phi_0^{-1}(S(F))$. Moreover, there exist  currents on $B$ with $L^1_{\loc}$-coefficients $\mathscr{T}_{\infty}(\omega)$ and $\mathscr{T}_{\infty}(d\omega)$ such that:
\begin{equation}\label{t13} \lim_{t\ra \infty}\phi_t^*\omega-\phi_0^*\omega=(-1)^{|\omega|}d[\mathscr{T}_{\infty}(\omega)]+\mathscr{T}_{\infty}(d\omega).
\end{equation}
\end{theorem}

Let us  take a more detailed look at the context. 

The precursor to Theorem \ref{t11} is Theorem 10.3 in  \cite{HL2}. This is the particular case where one takes $P$ to be the fiber bundle with fiber the Grassmannian of linear subspaces of a certain $\bZ_2$-graded vector bundle $E\oplus F$. The flow considered in that article is the compactification of the linear flow on $\Hom(E,F)$ which takes $A\ra tA$. The proof, which seems to contain a glitch (see Remark \ref{glrem})  takes advantage  of the fact that the fiber is a Grassmannian and does not generalize. Instead, the main techniques to prove Theorem \ref{t11} are contained in \cite{La} where  Latschev shows, among other things that the stable and unstable manifolds of a Morse-Bott-Smale flow in a \emph{compact} manifold have finite volume. The idea is to construct a compact manifold with corners which comes with a surjective projection to the closure of the stable/unstable manifold, projection which is a diffeomorphism on a dense open subset: in other words, a  Bott-Samelson resolution with corners. Notice that the finite volume of the unstable manifolds is necessary in relation (\ref{t12}) in order for the right hand side to make sense if no integrability condition is imposed on $\omega$. In order to prove Theorem \ref{t11} and its cousins we rewrote the proof of the main technical lemma from Latschev's article \cite{La} in order to fit our necessities. By the way, Latschev's homotopy formula    (Theorem 4.1 in \cite{La}) becomes a particular case of the above result for a certain tautological section (see Corollary \ref{corLat}).

The interest in  formulas like (\ref{t12}) and (\ref{t13})  comes from the application of Theorem \ref{t11}  to closed forms $\omega$ that arise from standard constructions. One example is the classical Chern-Weil theory with $\omega$ a characteristic form associated to a connection in a vector bundle $E\ra B$. Initially this is a form on the base space.     In order to get the machinery running, one needs another piece of data and that is a section of $E$. One can in fact  think of it as a section of $P:=\bP(\bC\oplus E)$. Now one pulls back the bundle with connection $(\bC\oplus E,d\oplus\nabla)$ over   $P$. The theory applies to the characteristic form of the Chern connection of the hyperplane tautological bundle $\tau\subset \bC\oplus E$, connection obtained by orthogonally projecting $d\oplus\nabla$ onto $\tau$. The example of the transgression formula for   the top Chern form of a connection on a complex vector bundle is worked out in complete detail in Section \ref{tCc}. While this is in fact a perfect example of the theory developed in \cite{HL2}, the details are spread out over two articles:  \cite{HL1} and  \cite{HL2}. We gather all the pieces here because this is a beautiful illustration of how one can go quite far armed with formula (\ref{t12}) and a minimal number of computations. In fact,  in this case, one has to perform just one universal computation. The computation seems standard but we include it here for completeness.

  The set-up  of a fiber bundle with a Grassmannian fiber together with the compactification of the linear flow  mentioned a bit earlier seems to be sufficient for such applications. However, there are other situations in which one would like to be able to flow forms at "infinity" and   concretely describe the resulting current. In principle, the number of applications is only limited by the ability of finding locally constant Morse-Bott-Smale flows (for the exact requirements see Section \ref{section 3}) and interesting forms $\omega$. We apply Theorem \ref{t11} in Sections \ref{occMs} and  \ref{SCcf} to produce transgression formulas in (cohomological) odd degree. We recover a result of Nicolaescu \cite{Ni3} and answer a question of Quillen from \cite{Qu0} with this occasion. 
  
   In his celebrated article \cite{Qu0}, Quillen left open the question of weak convergence of the Chern character forms built out of superconnections on a vector bundle. These Quillen superconnections are made up from a connection $\nabla$ and a self-adjoint endomorphism $A$ on a bundle. In \cite{Qu0} he proved that the Chern character forms for $tA$, in the limit  $t\ra \infty$, concentrate on the singular set of $A$, i.e. the points where $A$ has kernel.  Quillen asked about the exact conditions of when this limit exists.  While  it was  clear from the very first treatise on the subject (\cite{HL1}) that the characteristic forms Harvey and Lawson built using a section, a connection and a mode of approximation bears a  resemblance to the superconnection formalism, to our knowledge there has not appeared in the literature an exposition making this resemblance more apparent. We take up this task here in Section \ref{SCcf}, not through the theory of singular connections but  with the means at our disposal. We take advantage of another important result of Quillen \cite{Qu}, that allows the extension of the superconnection Chern character forms to accommodate situations when the bundle endomorphism is unitary and not self-adjoint as in the original construction. This is done using the Cayley and Laplace transforms in \cite{Qu} and applies equally well to even and odd $K$-theory. This "trick" gives forms defined  not only on a dense, open set but on the full Grassmannian,  forms which can be flown out. One of the highlights of our applications is the computation of this limit at infinity, described explicitly in Theorems \ref{ansQu} and \ref{ansQu2}.
  
  The main result of Nicolaescu in \cite{Ni3} is about the Poincar\'e duality of certain invariant forms  on the unitary group $U(n)$ (arising from the Maurer-Cartan connection) and certain currents appearing as stable manifolds of natural "height" Morse functions on $U(n)$. The original proof made appeal to the theory of analytic currents of Hardt \cite{Ha}. We should say that analyticity\footnote{Liviu Nicolaescu brought to our attention the fact that the model flows we considered in this paper  are either tame or conjugate to  tame flows (see \cite{Ni4} and Example \ref{univmodel}).} has its own important role in the theory as clearly demonstrated in \cite{HL2}. Nevertheless, it is quite straightforward to see that Nicolaescu's Theorem is a particular case of Theorem \ref{t11}. The hardest part seems to be the computation of a certain residue, a result which we believe is interesting in itself. This is done in Appendix \ref{resU}.

  In Section \ref{occMs}, we introduce a representative for a degree zero de Rham-Federer differential character called the Maslov spark. This object is naturally associated to a pair made of a unitary bundle isomorphism and a connection. The Maslov cycle, the locally integrally flat part of the Maslov spark,  and its higher degree relatives, discussed at length in  \cite{Ci}, have a well known relation with the spectral flow of a family of self-adjoint Fredholm operators. 
    
   In the last section we show  the usefulness of  Theorem \ref{t11} in its  non-compact formulation. The fiber bundle $P\ra B$ is simply a vector bundle.  We recover a  formula of Getzler from \cite{Ge}, equivalent with Chern-Gauss-Bonnet, formula where he uses Thom forms which are Gaussian shaped, in the spirit of the Mathai-Quillen formalism.
In fact, Getzler's result and a theorem by Bismut in \cite{Bi} and in general the superconnection formalism were the other important motivation for the present results.  

One important notion Harvey and Lawson introduced in \cite{HL2} is that of geometric atomicity. One says that a section $\phi_0$ is (weakly) geometrically atomic with respect to a vertical vector field $X$ on the fiber bundle $P\ra B$ if the volume of  $\phi ({[0,\infty)}\times B)$ is finite. It turns out that a  transgression formula \ref{t13} makes sense in this conditions even if $X$ is not the vertical gradient of a nice Morse-Bott-Smale function (see Theorem \ref{t1}). However, in this general context not much can be said about $\displaystyle\lim_{t\ra \infty}\phi_t^*\omega$. Geometric atomicity is philosophically central to the article since the convergence of currents (locally) in the mass norm or the flat norm depends essentially on having finite volume at finite distances. 
The completeness of the fibers, properness of $f$ and the integrability conditions on $\omega$ - a tuple being strongly atomic is essentially one such condition - insure that things behave well at $\infty$ as well. As a matter of fact, Theorem \ref{t11} is a two stage generalization of the aforementioned result from \cite{HL2}. In the first stage, we assume compact fiber and no condition on $\omega$ and we prove  Theorem \ref{refhomf}. In the second stage we noticed that not much changes if the finite volume condition is replaced by an integrability condition. This clarifies also the role of geometric atomicity, weak or strong.

As we said before, one big part of the article is made up of proofs of known results and one of these is Latschev's main theorem about the existence of Bott-Samelson resolutions with corners.  We include a complete treatment  in the Appendix because on one hand we needed the details in the proof of Theorem \ref{refhomf} and on the other hand because it brings a certain degree of simplification to some of the original arguments.

Among other applications, we give a short proof of a refined version of the  Chern-Gauss-Bonnet Theorem using the fiberwise one-point compactification of an oriented real vector bundle and the height function in the role of $f$. As opposed to most other known proofs this one does not go through the construction of a Thom form on the real vector bundle first. 

Finally, in Section \ref{occMs} we include a short discussion about a possible approach to an odd Chern-Weil theory.

\vspace{0.3cm}

\noindent
\emph{Acknowledgements:} My first contact with the fundamental work of Harvey and Lawson was through conversations with  Liviu Nicolaescu. He also made important suggestions that helped improved the exposition.  For this and for many other obvious reasons, I would like to warmly thank him.  In the past several years, several people have listen to me occasionally   talking about transgression and sparks and superconnections.  I can not say for sure that I convinced them I was talking about mathematics. Among them are Marianty Ionel, Michael Deutsch,  Jorge de Lira, Levi Lima, Lev Birbrair, Paulo Piccione and Luciano Mari. I owe them many thanks,  for patiently listening to my repeatedly imprecise statements.

\section{\bf{Vertical vector fields and geometric atomicity}}
\label{s2}

 All manifolds used in this article are smooth, finite dimensional  and \emph{without boundary}, although they could potentially be disconnected and non-compact unless stated otherwise. 

 Let $\pi:P\ra B$ be a locally trivial fiber bundle with fiber  $M$ and an $n$-dimensional, \emph{oriented} base manifold $B$. Suppose $P$ is endowed with a Riemannian metric. Let $X:P\ra VP$ be a vertical vector field on $P$ where the vertical tangent space $VP:=\Ker d\pi$ represents the collection of all the tangent spaces to the fibers. In the next section we will be imposing certain pleasant properties on $X$ in order to obtain refined results. Here we will only assume that the flow determined by $X$ is \emph{complete}, which is automatic for example if $F$ is compact. Notice that every integral curve of $X$ that starts in a fiber stays in the same fiber forever. Let $\Phi:\bR\times P\ra P$ be the flow of $X$.

If $s=\phi_0:B\ra P$ is a section, we consider the map:
\begin{equation}\label{flsec} \phi:\bR\times B\ra P,\quad\quad \phi (t,b)=\Phi(t,\phi_0(b)),
\end{equation}
which is the flowout of $s$.

Notice that for every integer $k\geq 0$ and for every $t\geq 0$ we have an operator 
\[ \eT_t:\Omega^k(P)\ra \Omega_{k-1}(B), \quad\quad \eT_t(\omega)=\left\{ \eta\ra \int_{[0,t]\times B}\phi^*\omega\wedge p_2^*\eta\right\}, \quad \forall \eta\in\Omega_0^{n-k+1}(B),
\]
where $\Omega^{k}(P)$ are smooth differential forms on $P$,  $\Omega_{k-1}(B)$ are currents of degree $k-1$ on $B$, meaning elements of the topological dual to $\Omega_0^{n-k+1}(B)$ (forms with compact support) and $p_2:\bR\times B\ra B$ is the projection on the second factor.
\begin{remark} Notice the different grading on $\Omega_*(B)$. We follow \cite{HL3} in using a convention   that has the advantage of making the operator $\partial$ of degree $+1$ just like $d$ on forms. This turns the natural inclusion $\Omega^*(B)\hookrightarrow \Omega_*(B)$ into a morphism of differential graded algebras.
\end{remark}
\begin{remark} The current  $\eT_t$ can be described simply by using standard operations with currents:
\[ \eT_t(\omega)=\pi_*(\omega\wedge \phi_*([0,t]\times B)).
\]
This holds essentially because $\pi\circ \phi=p_2$ which is due to the fact that $X$ is vertical.  We used $\wedge$ where it is standard to put $\llcorner$. This is due to the  grading convention.
\end{remark}
\begin{remark} Notice that in order for the integral above to make sense one only needs an orientation on $B$. For everything that follows one does not need $P$ or the fiber $M$ to be orientable, just for $B$ to have a fixed orientation.  One can in fact remove even this requirement by working with twisted forms as described in the Appendix of \cite{HL2}. In our set-up one would ask for $\omega$ to be an \emph{untwisted form} (for which the operation of pull-back makes sense)  but all $\eta$ would be twisted forms on $B$ and therefore $\phi^*\omega\wedge \eta$ would be twisted and could be integrated over $B$.   We do not see any inconvenient in using this convention throughout the article. However, we preferred a more classical set-up. 
\end{remark}
We investigate the conditions under which the limit $\displaystyle\lim_{t\ra \infty}\eT_t$ exists in the weak sense. Let us start with an elementary result.

Let $Z:=(X\circ \phi_0)^{-1}(0)\subset B$ be the singular locus of the section $\phi_0$ with respect to $X$. We have:
\begin{lemma}
\[ \eT_t(\omega)=\left\{ \eta\ra \int_{[0,t]\times B\setminus Z} \phi^*\omega\wedge p_2^*\eta \right\},\quad\quad\forall \eta\in\Omega_0^{n-k+1}(B).
\]
\end{lemma}
\begin{proof} Notice that by Fubini theorem 
\[ \mathcal{T}_t(\omega)(\eta)=\int_{B}\left(\int_{[0,t]}\phi^*\omega\right)\wedge\eta,
\]
where $\int_{[0,t]}\phi^*\omega$ is the integral of $\phi^*\omega$ over the fiber of the projection $p_2:[0,t]\times B\ra B$. Every $k$ form $\alpha\in\Omega^{k}([0,t]\times B)$ can be written uniquely as
\[ \alpha=\beta\wedge dt+\gamma
\]
such that the contraction of $\beta$ or $\gamma$ with $\frac{\partial}{\partial t}$ is zero. Notice  that $\beta=(-1)^{k-1}\iota_{\frac{\partial}{\partial t}}\alpha$.  We have
\[\iota_{\frac{\partial}{\partial t}}(\phi^*\omega)(s,\cdot)=\phi^*_s\left(\iota_{\frac{\partial \phi}{\partial t}(s,\cdot)}\omega\right)(\cdot)=\phi_s^*\left(\iota_{X(\phi(s,\cdot))}\omega\right)(\cdot),
\]
with $\iota_{\frac{\partial \phi}{\partial t}(s,\cdot)}\omega$ defined only for points in the image of $\phi$.
It follows that for a fixed $b\in B$ 
\[ \left(\int_{[0,t]}\phi^*\omega\right)(b)=(-1)^{k-1}\int_{[0,t]}\phi^*_s\left(\iota_{\frac{\partial \phi}{\partial t}(s,b)}\omega\right)(b)ds
\]
which is obviously $0$ if $b\in(X\circ\phi_0)^{-1}\{0\}$ since in that case $\frac{\partial \phi}{\partial t}(s,b)=0$ for all $s$.
\end{proof}

Now $\phi\bigr|_{[0,t]\times B\setminus Z}$ is obviously an immersion and we follow \cite{HL2} for the terminology we are about to introduce:
\begin{definition} A section $s:B\ra P$ is called weakly geometrically atomic on the positive/negative semi-axis with respect to $X$ if the set $[0,\infty)\times B\setminus Z$, respectively $(-\infty,0]\times B\setminus Z$ has \emph{locally} finite $(n+1)$-Hausdorff measure with respect to the Riemannian metric induced from $P$ via the immersion $\phi:\bR\times B\setminus Z\ra P$. This means that $\forall b\in B\setminus Z$ there exists a neighborhood $b\in U\subset B\setminus Z$ such that $\phi([0,\infty)\times U)$ has finite $(n+1)$-Hausdorff measure with an analogous property holding for the other case. 

The section $s:B\ra P$ is called weakly geometrically atomic  if it is so on both semi-axis.

\end{definition} 

Notice that if $s$ satisfies weak geometric atomicity on the positive semi-axis then $[0,\infty)\times (\supp \eta\setminus Z)$ will have finite Hausdorff measure for all $\eta\in\Omega^*_0(B)$. Most of the time we will use weak geometric atomicity on $[0,\infty)$ and therefore we will omit to specify the semi-axis. 

For the rest of this section \emph{weak geometric atomicity}  is the relevant notion. However, we will need a slightly stronger notion in Section \ref{Sect5}.

\begin{definition}\label{stratom} A section $s:B\ra P$ is called strongly geometrically atomic on the positive semi-axis with respect to $X$ if the set $[0,\infty)\times B\setminus Z$  has \emph{locally} finite $(n+1)$-Hausdorff measure with respect to the Riemannian metric induced from $P\times P$ via the immersion $\xi:\bR\times B\setminus Z\ra P\times P$ defined by:
\[ \xi(t,b)=(\phi_t(b),\phi_0(b)).
\]
\end{definition}
One could use $P\times_BP$ in  Definition \ref{stratom} but it would not make a difference since  $P\times_BP$ gets its natural metric from $P\times P$. The next easy result justifies the choice of words in the definitions above.
\begin{lemma} Strong geometric atomicity implies weak geometric atomicity. 
\end{lemma}
\begin{proof} The projection onto the first factor of $P\times P$ takes $\xi([0,\infty)\times B\setminus Z)$ diffeomorphically onto $\phi([0,\infty)\times B\setminus Z)$. This projection is a Riemannian submersion. It is an easy exercise to show that the volume of the section of a Riemannian submersion is bigger or equal the volume of the base space. Therefore
\[\Vol(\xi(\bR\times B\setminus Z))\geq\Vol(\phi(\bR\times B\setminus Z)).
\]
\end{proof}

\begin{remark} The definition of weak/strong geometric atomicity seems to depend on the metric on $P$. However, if $P$ is compact then one can check easily that the condition does not depend on the metric. It turns out that the following more general fact is true: the condition of geometric atomicity depends only on the equivalence class of vertical Riemannian metrics on the vector bundle $VP\ra P$, equivalence which can be defined as follows. Two metrics $g_1$ and $g_2$ are in the same class if for every $b_0\in B$ there exist a neighborhood $U$ and constants $C_1$ and $C_2$ such that $g_1(b)\leq C_1g_2(b)$ and $g_2(b)\leq C_2g_1(b)$, where $g_i(b)$ represent the Riemannian metrics on the fiber $P_p:=\pi^{-1}(b)$. In particular, any two vertical metrics are equivalent if the fiber $F$ is compact.  
\end{remark}
\begin{remark} Notice that the definition of geometric atomicity is vacuous for periodic flows since only sections $\phi_0$ for which $B\setminus Z$ has Hausdorff dimension smaller than $n$ could potentially fulfill the condition, a situation that will not be treated here. The prototype for $X$ is the gradient of a Morse function in each fiber.
\end{remark}

\begin{definition}  A smooth form $\omega\in\Omega^k(P)$ is called (locally) vertically bounded if for every compact $K\subset B$ there exists a constant $C=C(K)$ such that
\[|\omega_p|\leq C_K,\quad\quad\forall p\in \pi^{-1}(K).
\]
\end{definition}
In particular every smooth form on a fiber bundle with compact fiber is vertically bounded.

We have the following result:
\begin{prop}\label{p1} Let $\omega\in\Omega^k(P)$ be a vertically bounded, smooth differential form and $\phi_0:B\ra P$ be a weakly  geometrically atomic section with respect to $X$. Then
\[ \lim_{t\ra \infty}\eT_t(\omega)=\eT_{\infty}(\omega):=\left\{\eta\ra \int_{[0,\infty)\times B\setminus Z}\phi^*\omega\wedge p_2^*\eta\right\}.
\]
where the limit is in the topology  of currents with locally finite mass on $B$, i.e.  $\gamma T_t\ra\gamma T_{\infty}$ in the mass norm for every $\gamma$ a smooth function with compact support.  
\end{prop}
\begin{proof} The right hand side is well-defined because the form $\phi^*\omega\wedge p_2^*\eta$ is bounded on the set of finite volume $[0,\infty)\times \supp{\eta}\setminus Z$. It is also a current of locally  finite mass. Then $\eT_t(\omega)$ converges to $\eT_{\infty}$ by the Lebesgue dominated convergence theorem.  
\end{proof}
\begin{remark} The current $\eT_{\infty}(\omega)$ is represented by a form on $B$ with $L^{1}_{\loc}$-coefficients. By this we understand the following. Suppose we have a Riemannian metric on $B$. One can show that $\eT_{\infty}(\omega)(\eta)=\int_B \tilde{\omega}\wedge \eta$, where $\tilde\omega: B\ra \Lambda^{k-1}T^*B$ is a form which satisfies the property that for every $b\in B$ there exists a neighborhood $U$ such that for every smooth $\overrightarrow{\xi}: U\ra \Lambda^{k-1}TB\bigr|_{U}$ with $\|\overrightarrow{\xi}(b)\|=1$, $\forall b\in U$,  the function $b\ra \tilde{\omega}(\overrightarrow{\xi})(b)$ is integrable on $U$.

In fact,
\[ \tilde{\omega}(b)=\int_0^{\infty}\phi^*_s(\iota_{X(\phi(s,b))}\omega)(b)~ds,
\]
and Fubini Theorem for
 \[\int_{[0,\infty)\times \supp {\eta}}\phi^*\omega\wedge p_2^*\eta\left(\frac{\partial}{\partial t }\wedge\overrightarrow{\xi}\wedge\overrightarrow{\xi}^c\right)dt\otimes d\mathcal{H}^n\]
 for various choices of $\eta$ implies the claimed property for $\tilde{\omega}$. It should be clear that having $L^1_{\loc}$-coefficients does not depend on the choice of metric on $B$.
\end{remark}
\begin{remark} If $d\omega$ is also vertically bounded then $\eT_t(\omega)$ is locally normal by Stokes Theorem, and since the limit is taken locally in the mass norm it implies that $\eT_{\infty}(\omega)$ is locally flat. In particular this happens for $\omega=d\alpha$ with $d\alpha$ vertically bounded.
\end{remark}
\begin{remark} The underlying reason for which $\eT_{\infty}(\omega)$ might not be smooth is the fact that $\phi:\bR_{\geq}\times B\ra B$ is not a fiber bundle in general since $X$ has stationary points. However,   if  the vector field $X$ is a locally constant Morse-Bott-Smale vector field (see Def. \ref{MBSdef}) and  if the  image of $\phi_0$ is contained in the stable bundle $S(F)$ of a single stationary manifold $F$ of $X$ then $\eT_{\infty}(\omega)$ is smooth too. However, this is a  non-generic situation. 
\end{remark}

\begin{definition} The boundary operator for $\eT_t$ is
\[ \partial \eT_t:\Omega^*(P)\ra \Omega_*(B),\quad\quad \partial \eT_t(\omega)=\eT_t(d\omega)+(-1)^{|\omega|}d[\eT_t(\omega)],
\]
where $dT(\eta):=T(d\eta)$ for all currents $T\in \Omega_*(B)$.
\end{definition}
Unwinding the definition we get:
\[ \partial \eT_t(\omega)(\eta)=\int_{[0,t]\times B}d(\phi^*\omega\wedge p_2^*\eta)=\int_B \phi_t^*\omega\wedge \eta-\int_B\phi_0^*\omega\wedge\eta.
\]
\begin{theorem}\label{t1} Let $P\ra B$ be a fiber bundle and $X$ be a vertical vector field. Let $\omega\in\Omega^k(P)$ be a vertically bounded form such that $d\omega$ is also vertically bounded. Assume $\phi_0$ is weakly geometrically atomic with respect to $X$. Then 
\[ \lim_{t\ra \infty}\partial\eT_t(\omega)=\partial\eT_{\infty}(\omega),
\]
locally as flat currents on $B$.
\end{theorem}
\begin{proof} The limit $\lim_{t\ra \infty}\eT_t(d\omega)=\eT_{\infty}(d\omega)$ holds in the mass topology and $\eT_{t}(d\omega)$ are locally normal currents hence the limit holds also in the flat topology (see Federer 4.1.17). 

For the other term, notice that $d[\eT_t(\omega)]$ is locally normal due to the boundedness of $d\omega$ and $\eT_t(\omega)\ra \eT_{\infty}(\omega)$ in the mass norm hence $d[\eT_t(\omega)]\ra d[\eT_{\infty}(\omega)]$ in the flat norm, due to the obvious inequality:
\[ {\bf F}(dT)\leq {\bf M}(T).
\]
\end{proof}
\begin{corollary}\label{cltran}   If $P\ra B$ has compact fibers and $\omega\in\Omega^k(P)$ is closed, the following transgression formula of closed currents holds:
\[ \lim_{t\ra\infty} \phi_t^*\omega-\phi_0^*\omega=(-1)^{|\omega|}d[\eT_{\infty}(\omega)],
\]
and  $\displaystyle\lim_{t\ra \infty} \phi_t^*\omega$ is a locally flat, closed current on $B$. 
\end{corollary}
\begin{remark} This section could have been written almost entirely in the more general context of fiber bundles $\eP\ra B$ with fiber a Banach manifold $\eF$. We  return to these aspects in \cite{Ci2}. 
\end{remark}

In order to obtain more refined results about the limits of currents above we need to restrict our attention to nice vector fields $X$.

\section{\bf {Locally Morse-Bott-Smale vector fields and normal sections}}\label{section 3}

In \cite{HL3}, Harvey and Lawson introduce a class of gradient Morse flows on a compact  manifold $M$, that satisfy certain desirable properties which imply that 
\begin{equation}
\label{HLMS} \lim_{t\ra \infty} \phi_t^*\alpha=\sum_{p} \left(\int_{U_p} \alpha\right)\cdot [S_p],
\end{equation}
where $\phi_t:M\ra M$ is the flow, $\alpha$ is differential form on $M$ and $U_p/S_p$ are the unstable/stable manifold of a critical point $p$.  One of these desirable properties which in particular gives meaning to (\ref{HLMS}) is that $S_p$ has finite volume for all critical $p$ which  (together with an orientation) turns it into a current $[S_p]$ on $M$. Moreover they prove that these properties are satisfied by the gradient flow of a Morse function for which the stable and unstable manifolds of every pair of critical points intersect transversely, the so called Morse-Smale flows. Not long after,  Latschev in \cite{La} generalized their result to Morse-Bott  gradient flows satisfying a natural  transversality condition a la Smale.  Let us recall some definitions.

Let $M$ be a Riemannian manifold not necessarily compact and $f:M\ra \bR$ a smooth function whose gradient generates \emph{a complete} gradient flow: $\phi:\bR\times M\ra M$.
\begin{definition} The function $f$ is called:
\begin{itemize}
\item[(a)] Morse-Bott if the critical set of its gradient $\nabla f$ is a union of disjoint submanifolds and the Hessian  $\Hess_pf$ at a point $p$ in a critical manifold $F$ is non-degenerate on the orthogonal complement of $T_pF$.  Denote by $S(F)$ and $U(F)$, respectively the stable and unstable sets of $F$ which are unions of stable/unstable sets of points:
\[ S(F):=\{ p\in M ~|~\lim_{t\ra \infty}\Phi(t,p)=q \in F \}\qquad  S(F)=\bigcup_{q\in F} S(q) \]
\[ U(F):=\{p\in M ~|~\lim_{t\ra -\infty} \Phi(t,p)=q\in F\} \qquad U(F)=\bigcup_{q\in F} U(q).\]
\item[(b)] Morse-Bott-Smale if it has the extra property that for any two critical manifolds $F$ and $F'$ and for any $p\in F$ and $q\in F'$ the following manifolds are transversal:
\[ U(p)\pitchfork S(F')\qquad S(q)\pitchfork U(F).
\]
\end{itemize}
\end{definition}
One can speak of manifolds instead of just sets in the previous definitions because of the following result (see \cite{AB}, Proposition 3.2):
\begin{theorem} The stable and unstable manifolds $S(F)$ and $U(F)$ are images of injective immersions: $\eS:\nu^s(F)\ra M$ and $\eU:\nu^u(F)\ra M$, where $\nu^{s/u}(F)$ are those bundles over $F$ resulting from the decomposition of the $\Hess f\bigr|_{TF^{\perp}\times TF^{\perp}}$ into its negative and positive eigenspaces. Moreover the endpoint maps $S(F)\ra F$ and $U(F)\ra F$
\[  p\ra \lim_{t\ra \infty} \Phi(t,p)\; \; \mbox{and}\;\; p\ra \lim_{t\ra -\infty}\Phi(t,p)
\]
are smooth and restricted to a neighborhood of $F$ have the structure of locally trivial fiber bundles.
\end{theorem}

The other fundamental result about Morse-Bott functions is the existence of nice coordinates around critical manifolds a result better known as the Morse-Bott lemma:
\begin{lemma}\label{MBL} Let $f:M\ra \bR$ be a Morse-Bott function and let $F$ be a connected component of the critical set of $f$. For any $p\in F$ there exists a local chart of $M$ around $p$ and a local splitting of the normal bundle of $F$:
\[ \nu F=\nu^+(F)\oplus \nu^{-}(F),
\]
such that $f$ in these coordinates assumes the form:
\[ f(x,y,z)= c+\frac{1}{2}(|x|^2-|y|^2),\quad\quad \forall z\in B, x\in \nu^{+}(F), y\in \nu^{-}(F).
\]
\end{lemma}
\begin{remark} It is good to be aware of the fact that a  Morse-Bott function can not always be turned into a Morse-Bott-Smale function by a change of metric as is the case for classical Morse functions. One illuminating example can be found in \cite{La}, Remark 2.4. 
\end{remark}

We now return to the set-up of the previous section. Suppose $P\ra B$ is a fiber bundle with a Riemannian metric on $P$.  We will assume that for every $b\in B$ there exists a function $f:P_b\ra\bR$  such that the gradient flow of $f$ with respect to the metric induced from $P$ is  Morse-Bott-Smale   and $X\bigr|_{P_b}=\grad f$. This is the vertical "niceness" of $X$.

 We will require that $X$ be  locally trivial in the horizontal directions.   This means that around every $b\in B$ there exists a chart $U$ and a local trivialization $P\bigr|_{U}\simeq U\times M$ with corresponding trivialization $VP\bigr|_{U}\simeq U\times TM$, such that in these coordinates:
\[ X(u,p)=(u, \tilde{X}(p)),\quad\quad \forall u\in U, p\in M.
\]
the important point being that $\tilde{X}$ does not depend on $u$.
\begin{definition}\label{MBSdef} A vertical vector field $X:P\ra VP$ is called locally Morse-Bott-Smale if in every fiber it is gradient Morse-Bott-Smale and if it satisfies the horizontal local triviality condition described  above.
\end{definition}

\begin{example}\label{univmodel} One universal method to construct  flows as above is as follows. Suppose that on the manifold $M$ we fix a function $f$ whose gradient $X$ is Morse-Bott-Smale. Take $G\subset \Diffeo(M)$ to be a finite dimensional Lie subgroup of diffeomorphisms of $F$ that commute with the flow diffeomorphisms $\phi_t$ of $X$, i.e.
\[ \psi\in G\Leftrightarrow \psi \phi_t=\phi_t\psi \qquad\forall t
\]
Infinitesimally, this can be written as $d_p\psi(X_p)=X_{\psi(p)}$ for all $p\in F$. Now take $P\ra B$ to be a fiber bundle with fiber $F$ and structure group $G$. Then  $X$ induces a locally constant Morse-Bott-Smale vector field on $P$ as follows. Consider a principal bundle $\tilde{P}$ with structure group $G$ such that the associated bundle via the natural action of $G$ on $M$ is $P$. Consider the vector field on $\tilde{P}\times F$ which is constant in the $\tilde{P}$ variable and equal to $X$ in the $M$ variable. Due to the invariance property mentioned before this vector field descends to a vector field on $P$.

One parameter families of invariant diffeomorphisms $\psi$  arise  via  (complete) vector fields $Y$ such that $[X,Y]=0$. The flow diffeomorphisms of $Y$ satisfy the above property.
\end{example}

While we only require that $X$ be  locally a vertical gradient vector field, in all the examples we list below $X$ is in fact a vertical gradient of a globally defined function $f$.

\begin{example} Let $E,F\ra B$ be two complex vector bundles on $B$ with $E$ of rank $k$. Suppose they are endowed with hermitian metrics. Let $\Gr_k(E\oplus F)\ra B$ be the bundle of Grassmannians of subspaces inside $E\oplus F$ of the same dimension as the rank of $E$. The metrics on $E$ and $F$ induce a vertical metric on $\Gr_k(E\oplus F)$. Let $f:\Gr_k(E\oplus F)\ra \bR$ be defined by
\[ f(L)=\Real\Tr(\epsilon P_L),
\] 
where $\epsilon:E\oplus F\ra E\oplus F$ is the reflection in $E$ and $P_L:E\oplus F\ra E\oplus F$ is the orthogonal projection onto $L$. The fiberwise gradient of $f$ is a vector field of the type we described above. 
\end{example}
\begin{example} Let $E\ra B$ be a real vector bundle with a Riemannian metric and let $P:=S(\bR\oplus E)\ra B$ be the spherical bundle of $\bR\oplus E$. Let $f$ be the Morse function which is the restriction to $P$ of the projection on the first coordinate of $\bR\oplus E$. The vertical  gradient flow of  $f$ is the "height function" gradient flow in each fiber. 
\end{example}
\begin{example} Let $U(E)\ra B$ be the fiber bundle with fiber the unitary endomorphisms of a hermitian bundle $E\ra B$. The function
\[f(U)=\Real\Tr(U),
\]
is a function whose fiberwise gradient satisfies the condition above. More generally, if there exists $A\in\End(E)$, self-adjoint such that $A$ induces a decomposition of $E$ into  eigenbundles $E=E_1\oplus E_2\oplus\ldots \oplus E_k$ in such a way that $A\bigr |_{E_i}\equiv \lambda_i \id_{E_i}$ then
\[ f(U)=\Real\Tr (AU),
\] 
is also a a locally constant Morse-Bott-Smale function.
\end{example}

The local triviality condition for $X$ plays a crucial role. It allows to write $X$ as a  Morse-Bott-Smale gradient vector field on open subsets of $P$. Indeed one can use the local trivialization $P\bigr|_{U}\simeq U\times M$ and consider the Morse-Bott-Smale vector field $\tilde{X}=\grad{f}$ on $M$. Then for any choice of horizontal metric on $U$,  the vector field
\[ (u,p)\ra (u,\tilde{X}(p))
\]
is the gradient of $\tilde{f}(u,p)=f(p)$, provided one puts the direct sum metric on $TP\bigr|_{U}=VP\oplus HP$.  Hence one can make use of the technique introduced by Harvey and Lawson in \cite{HL3} and more generally by Latschev in \cite {La} for  the gradient of $\tilde{f}$ (see Appendix \ref{appA} for details).
\begin{definition} Let $\phi_0:B\ra P$ be a section. It is called \emph{s-normal} with respect to $X$ if   $\phi_0$ is transversal to all stable manifolds $S(F)$ of $X$.

 A section is called \emph{u-normal} if a dual statement holds with respect to the unstable manifolds. 

A section which is both s-normal and u-normal will be called normal. 
\end{definition}
The next result is inspired by Proposition 9.4 in \cite{HL2}. However, the proof presented in that article does not generalize to this more general situation.
\begin{remark} \label{glrem} In fact, we were not able to find a justification for the unproved claim made in Lemma 10.1 of the same article  (on which Proposition 9.4 is based) regarding certain coordinates  with respect to which the section is of \emph{product type}.  
\end{remark}
\begin{theorem}\label{gan} Let $P\ra B$ be a fiber bundle with compact fiber and a locally Morse-Bott-Smale vertical vector field $X$. An s-normal section is strongly geometrically atomic on the positive semi-axis with respect to $X$. Dually, an u-normal section is strongly geometrically atomic on the negative semi-axis.
\end{theorem} 


\begin{proof} We use the same technique Latschev \cite{La} used to prove that any stable or unstable manifold of a Morse-Bott-Smale flow inside a compact manifold has finite volume. For the convenience of the reader, Appendix \ref{appA} contains a detailed exposition of these techniques. 

  Notice first that strong geometric atomicity of $\phi_0:=s$ is just weak geometric atomicity for the section $\xi_0:B\ra P\times_BP$ with respect to the vertical vector field $(X,0)$ (see Example \ref{exlahl}). This vector field is locally Morse-Bott-Smale if $X$ is and the $s$-normality of $\phi_0$ implies the $s$-normality of $\xi_0$. Hence, in what follows, we will prove weak geometric atomicity for $\phi_0$ with the understanding that  $P$  plays the role of $P\times_BP$ and $\phi_0$ that of $\xi_0$. 

The statement is local in $B$ so we can fix $b_0\in B$ and look at a small enough neighborhood $U$ of $b_0$. We distinguish two cases:

\vspace{0.3cm}
\noindent
\emph{Case I:} The point $\phi_0(b_0)$ is not stationary for $X$. We can take $b_0\in U$ such $\phi_0(b)$ is not stationary for $X$ for all $b\in U$. 

We can also assume without loss of generality that $\phi_0\bigr|_{U}\subset f^{-1}(f(\phi_0(b_0))-\delta)$ for some $\delta>0$. Indeed, let $\delta>0$ be such that $f(\phi_t(b))>f(\phi_0(b_0))-\delta$ for all $b\in U$ and $t\geq 0$ and there is no critical level in between $f(\phi_0(b_0))$ and $f(\phi_0(b_0))-\delta$.   Let $\sigma:U\ra f^{-1}(f(\phi_0(b_0))-\delta)$ parametrize the manifold $\phi ((-\infty,0]\times U)\cap f^{-1}( f(\phi_0(b_0))-\delta)$ and let $\Sigma$ be the flowout of $\sigma(\bar{U})$.  Notice that $\phi([0,\infty)\times U)\subset \Sigma$. 

The techniques presented in Appendix \ref{appA} apply in order to construct a compact manifold with corners $\tilde{\Sigma}$ and a smooth map $\Pi:\tilde{\Sigma}\ra P$ such that  $\Imag\Pi=\overline{\Sigma}$ and $\Pi$ is a diffeomorphism from the interior points to an open dense set of $\Sigma$.    Here is where the transversality condition on $\phi_0$ is crucial (Lemma \ref{transvlemma}).

\vspace{0.3cm}

\noindent
\emph{Case II:} The point $\phi_0(b_0)$ is stationary for $X$. This is the more challenging one.

 Suppose without loss of generality that $\phi_0(b_0)\in F$, a critical manifold that lies entirely at energy level $0$ and that $\delta>\phi_0(b)>-\delta$ for all $b\in U$  for a $\delta>0$ small. We consider the nice local coordinates that the Morse-Bott Lemma \ref{MBL} provides.
 
  We  break $\Theta:=\overline{\phi([0,\infty)\times U)}$ into two pieces: $\Theta_{\geq\delta}:=\Theta\cap f^{-1}([\delta,\infty))$ and $\Theta_{\leq\delta}:=\Theta\cap f^{-1}([-\delta,\delta])$ and show  that both $\Theta_{\leq \delta}$ and $\Theta_{\geq \delta}$ have finite $(n+1)$-Hausdorff measure. Let $\Theta_{\delta}:=\Theta\cap f^{-1}(\delta)$.
 
We claim that  there exists a manifold with corners $\widetilde{W}$ of dimension $n$ and a map $\theta_{\delta}:\widetilde{W}\ra f^{-1}(\delta)$ transversal to all stable bundles of critical manifolds such that $\Theta_{\delta}\subset \Imag \theta_{\delta}$. For this end, let us write the map $\phi_0$ in local coordinates around $b_0$:
\[ \phi_0:\overline{B_{\bR^k}(0,1)}\times \overline{B_{\bR^{n-k}}(0,1)}\ra \bR^{k}\times \bR^{m}\times {\bR^p},\quad\quad \phi_0(a,b)=(a,\alpha(a,b),\beta(a,b)),
\]
where we chose the local coordinates around $b_0$ such that $p_1\circ \phi_0$ is the projection on the first factor of $\bR^{k}\times \bR^{n-k}$. This is possible because the transversality of $\phi_0$ with the stable manifold of $F$ is equivalent with $p_1\circ\phi_0$ being a submersion ($p_1:\bR^{k+m+p}\ra \bR^k$ is the projection).  By shrinking the neighborhood around $b_0$ even more we can assume that $|a|\cdot |\alpha(a,b)|\leq \epsilon$. 

We will keep in the back of our mind the fact that $p\geq n$ and the first  component of $\beta:\bR^n\ra \bR^n\times \bR^{p-n}$ is the identity since $\phi_0$ is a section and also due to the local triviality of the flow. 

  Let $\widetilde{W}:=[0,1]\times S^{k-1}\times \overline{B_{\bR^{n-k}}(0,1)}$ be the total space of the blow-up of $\{a=0\}$. Define (see also \ref{eqap1})
  \[ V:=\{(x,y,z)\in \bR^{k+m+p}~|~-2\delta\leq |x|^2-|y|^2\leq 2\delta,\;\; |x|\cdot |y|\leq \epsilon\}
  \] 
  \[ V_{\delta}:=\{(x,y,z)\in V~|~ |x|^2-|y|^2=2\delta,\;\; |x|\cdot |y|\leq \epsilon \}\]For $\delta>0$, let $\theta_{\delta}:\widetilde{W}\ra V_{\delta}:$ 
\begin{equation}\label{eq2t31} \theta_{\delta}(\lambda,v,b):=\left( v\sqrt{\sqrt{\delta^2+\lambda^2 |\alpha(\lambda v,b)|^2}+\delta},\frac{\lambda\alpha(\lambda v,b) }{\sqrt{\sqrt{\delta^2+\lambda^2 |\alpha(\lambda v,b)|^2}+\delta}}, \beta(\lambda v,b)\right).
\end{equation}
which is obtained as follows.  One  has a map:
 \[\chi_{\delta}: V\setminus \{x=0\}\subset \bR^{k+m+p}\ra V_{\delta}\]
  which takes a point $p\notin S(F)$ to the intersection of the flowline $p$ determines with the level set $f^{-1}(\delta)$ (see Remark \ref{flowlinepsi}).  For $\lambda\neq 0$, the map $\theta_{\delta}$ is  the composition of $\chi_{\delta}$ with $\phi_{0}\bigr|_{\{a\neq 0\}}$ and with  the blow up projection $(\lambda,v,b)\ra (\lambda v,b)$.  The expression (\ref{eq2t31})  makes it clear that $\theta_{\delta}$ extends smoothly to $\lambda=0$. 

From the above description we have that the image of $\theta_{\delta}$ is compact and contains the intersection of $\phi([0,\infty)\times U)$ with $f^{-1}{(\delta)}$. It follows that $\Theta_{\delta}\subset \Imag \theta_{\delta}$ because  $\overline{\phi([0,\infty)\times U)}$ is made up of (possibly) broken trajectories and each such trajectory is a limit of unbroken ones. Therefore we have in fact:
\[ \Theta_{\delta}=\overline{\phi([0,\infty)\times U)\cap f^{-1}(\delta)}.\]

The map $\theta_{\delta}$ is transversal to the other stable bundles of critical manifolds (see  Lemma \ref{transvlemma}) and therefore Latschev's blow-up technique gives us again that the flowout of $\Imag \theta_{\delta}$ has a "resolution" in the form of a compact manifolds with corners. 

In order to show that $\Theta_{\leq \delta}$ has finite volume, take another look at $\phi_t$ in local coordinates:
\[ \phi_t(a,b)=(e^ta,e^{-t}\alpha(a,b),\beta(a,b)),
\]
Taking $t=-\ln (s)$ for  $s>0$ we get that locally around the origin the flow is described by the image of the map:
\[ (s,a,b)\ra (s^{-1}a, s\alpha(a,b),\beta(a,b)).
\]
We precompose this with the change of variables $(s,a,b)\ra (s,sa,b)$ to get that the flow is described by the image of the differentiable  map:
\begin{equation}\label{eq3} (s,a,b)\ra(a, s\alpha(sa,b),\beta(sa,b)),
\end{equation}
which makes sense also for $s=0$. The condition that this image be contained in $f^{-1}([-\delta,\delta])$ implies that
\[ |a|^2-s^2\alpha(sa,b)\leq 2\delta.
\]
We can assume that $\alpha$ is bounded to begin with ($\phi_0$ is defined on a compact set)  and since $(a,b)$ varies within a compact set we get from the previous relation that there exists an upper bound on $s$. Therefore $\Theta_{\leq \delta}$ is in the image of the differentiable map (\ref{eq3}) defined on a compact set, say $[0,M]\times K$. Clearly  this image will have finite $(n+1)$-Hausdorff measure and this concludes the proof.

\end{proof}

We close this section with an obvious extension of Theorem \ref{gan}. We only needed compactness of the fiber in the above proof in order to guarantee that the sequence of blow-ups terminates and we end up with a compact manifold with corners. However, one does not need the full compactness of the fiber for this to work. One just needs  that the function $f$ that gives the gradient flow on the fiber has a maximum since everything will flow to the maximum critical manifolds of $f$. The rest is as before.
\begin{theorem} The first result of Theorem \ref{gan} is valid also for fiber bundles $P\ra B$ with non-compact fiber by requiring that the vertical vector field in one fiber is the gradient of a Morse-Bott-Smale function which is \underline{bounded above}. Similarly, the second result is valid for Morse-Bott-Smale functions bounded below.
\end{theorem}

 This suggests that the  applicability of Theorem \ref{gan} depends only on the gradient vector field in the fiber more than on the fiber itself. This does  not come as a surprise since the topology of the  manifold \emph{is determined} by the gradient of a Morse-Smale function.

 In Section \ref{Secnc} we discuss a  non-compact extension of Theorem \ref{gan} which is less trivial.

\section{\bf {A refined homotopy formula}}\label{Sect5}

Let  $\pi:P\ra B$ be  a fiber bundle of oriented manifolds  over an $n$-dimensional manifold $B$  endowed with a  vertical vector field $X\in\Gamma(VP)$ which is locally constant  Morse-Bott-Smale as in Section \ref{section 3}. Suppose that the function $f$ with vertical gradient  $X$  is bounded above. In particular, this happens if the fiber is compact. 

The plan is to prove the next result. 
\begin{theorem}\label{refhomf} Let $\phi_0:B\ra P$ be an s-normal section with respect to the field $X$ and let $\omega$ be a  form on $P$ of degree $k\leq n$.  Let $\tau_F:\phi_0^{-1}(S(F))\ra F$ be the composition  of the stable bundle projection $S(F)\ra F$ with $\phi_0$. The following identity of locally flat currents in $B$ holds:
\[\lim_{t\ra \infty} \phi_t^*\omega=\sum_{\codim{S(F)\leq k}}\Res_F^u(\omega)[\phi_0^{-1}(S(F))]
\]
where $\Res_{F}^u(\omega):=\tau_F^*\left(\displaystyle\int_{U(F)/F}\omega\right)$.
\end{theorem}

We will need first a few preparatives that insure that the expression on the right hand side makes sense as a current. The issues at stake are taken care of by the following two results:
\begin{prop} \label{essprop1} Let $M$ be a Riemannian manifold and $f$ a Morse-Bott-Smale function on it which is bounded above. Then the fibers of the projection $U(F)\ra F$ have finite volume, for every critical manifold $F$.
\end{prop}
\begin{proof} The proof is analogous to the proof of Latshcev's Theorem \ref{Lath}, with the techniques detailed in the Appendix \ref{appA}.
\end{proof}

\begin{prop} \label{essprop2} Let $\phi_0:B\ra P$ be an s-normal section with respect to the locally constant Morse-Bott-Smale vertical vector field $X$. Then for every critical manifold $F$, $\phi_0^{-1}(S(F))$ has locally finite volume in $B$ when $B$ is endowed with the pull-back metric $\phi_0^*g$. 
\end{prop}
Notice that the conclusion of Proposition \ref{essprop2} does not depend on the metric on $B$. However we have to fix a metric for definiteness. 
\begin{proof}  The plan is to show that around each point $b\in B$, each of the sets $\phi_0^{-1}(S(F))$ has finite volume. The main idea  is that after doing the same blow-ups as in  the proof of Theorem \ref{gan}, the strict transform of each of the sets $\phi_0^{-1}(S(F))$ is the preimage of a manifold via a \emph{completely transversal} smooth map  whose domain is a compact manifold with corners. The crucial point here is the Smale property of the flow as appears in the proof of Lemma \ref{transvlemma}.

Suppose that $\phi_0(b)\in S(F)$ for some $F$. Just like in the proof of Theorem \ref{gan}, we make certain simplifying assumptions. We choose $U$, a slice chart for $\phi_0^{-1}(S(F))$ such that $\phi_0(U)$ lands in a chart where we can use the coordinates of the Morse-Bott Lemma. In particular,   $\phi_0(U)$  intersects only $S(F')$ where $\overline{S(F')}\supset S(F)$ which  implies in particular that $\codim{S(F')}\leq \codim{S(F)}$. 

We distinguish again  two cases:

\vspace{0.3cm}
\noindent
\emph{Case I:}  The point $\phi_0(b)\in f^{-1}(-\delta)\cap S(F)$ is not critical.  

We can also assume without restriction of the generality that $\phi_0(U)\subset f^{-1}(-\delta)$ (with $F$ at energy level $0$). This is because the diffeomorphism induced by the flow which  takes $\phi_0(U)$ to $f^{-1}(-\delta)$ will take $\phi_0(\phi_0^{-1}(S(F))\cap U)$ to a submanifold of $ f^{-1}(-\delta)$ with comparable volume since this diffeomorphism has bounded differential. 

Now $\phi_0^{-1}(S(F))\cap U$ is a neighborhood of $b$ in the manifold $\phi_0^{-1}(S(F))$  hence it has finite volume. The non-trivial question is about  $\phi_0^{-1}(S(F'))$ with $S(F)\subset \overline{S(F')}$. Let $\Bl:\tilde{U}\ra U$ be the blow-up of $\phi_0^{-1}(S(F))$ inside $U$.  We use again the neighborhood $V$ of $F$ (see (\ref{eqap1})) and $V_{-\delta}=V\cap f^{-1}(-\delta)$. 

Due to the transversality of $\phi_0$ with $S(F)$, one has a map  (see \ref{apcommdiag})
\[ \tilde{\phi}_0:\tilde{U}\ra \tilde{V}_{-\delta}\quad \mbox{ such that }\quad \Psi_{-\delta}\circ \tilde{\phi}_0=\phi_0\circ \Bl\]
where $\Psi_{-\delta}:\tilde{V}_{-\delta}\ra V_{-\delta}$ is the blow-up along $B_{-\delta}$ described in (\ref{apcommdiag1}).   One also has a map $\tilde{\phi}_{0,\delta}:\tilde{U}\ra \tilde{V}_{\delta}$, which is basically $\tilde{\phi}_0$, except for the first coordinate (see (\ref{apdiag2})). The map relevant for the discussion is
\[ B_{\phi_0}:=\Psi_{\delta}\circ\tilde{\phi}_{0,\delta}:\tilde{U}\ra V_{\delta}.
\]
In Lemma \ref{transvlemma}, it is proved that if $\phi_0$ is transversal to $S(F')\cap f^{-1}(-\delta)$ then $B_{\phi_0}$ is transversal to $S(F')\cap f^{-1}(\delta)$ as well. 
The main observation is that the strict transform of $\phi_0^{-1}(S(F'))$ is $B_{\phi_0}^{-1}(S(F'))$ for each $F'$ for which the following relation holds:
 \begin{equation}\label{eq2s4}\overline{\phi_0^{-1}(S(F'))}=\phi_0^{-1}(S(F'))\cup \phi_0^{-1}(S(F)). \end{equation} The strict transform is by definition the closure of $\Bl^{-1}(\phi_0^{-1}(S(F')))$ in $\tilde{U}$. To see  that
 \begin{equation}\label{eq2s5} \overline{\Bl^{-1}(\phi_0^{-1}(S(F')))}=B_{\phi_0}^{-1}(S(F'))
 \end{equation}  holds, take a sequence of points   $x_n\in \phi_0(\phi_0^{-1}(S(F')))\cap f^{-1}(-\delta)$ such that $x_n\ra x\in S(F)\cap f^{-1}(-\delta)$. The (simple) trajectories that each $x_n$ determines \footnote{trajectories that end at $F'$} get close to a   trajectory which is broken only once at a point in $F$ and which ends up at $F'$ as well. Checking the exact definition of the map $B_{\phi_0}$ it is not hard to see that this type of broken trajectories are in one-to-one correspondence with points in $B_{\phi_0}^{-1}(S(F'))\cap \Bl^{-1}(\phi_0^{-1}(S(F)))$. It also is easy to see that away from the exceptional divisor $\Bl^{-1}(\phi_0^{-1}(S(F)))$ the two sides of (\ref{eq2s5}) coincide.

The set equality (\ref{eq2s5}) implies that $\phi_0^{-1}(S(F'))$  has locally finite volume for each $F'$ such that (\ref{eq2s4}) holds since the closure of $\Bl^{-1}(\phi_0^{-1}(S(F')))$ is a compact manifold with corners and the blow-up map $\Bl$ takes the interior of this manifold with corners to the $\phi_0^{-1}(S(F'))$. Now this exceptional divisor becomes the new blow-up locus at the next stage in which all $\phi_0^{-1}(S(F'))$ for $F''$ which satisfy:
\[ \overline{\phi_0^{-1}(S(F''))}=\phi_0^{-1}(S(F''))\cup \overline{\phi_0^{-1}(S(F'))}.
\] 
By induction, one proves that all $\phi_0^{-1}(S(F'))$ have finite volume. The process terminates because $f$ is bounded above.

\vspace{0.3cm}
\noindent
\emph{Case II:} The point $\phi_0(b) \in F$ is  critical.  

The transversality of $\phi_0$ with $S(F)$ ensures at least that $\phi_0^{-1}(S(F))$ has locally finite volume. In order to prove that the other $\phi_0^{-1}(S(F'))$ have finite volume one uses again the map  (\ref{eq2t31}) to put oneself in the same situation as in Case I,  the role of $U$ being played by a manifold with boundary.
\end{proof}
The next result of Federer is essential for the proof of Theorem \ref{refhomf}. We state it using the Hausdorff measure instead of the original, stronger version using the integral-geometric measure. 
\begin{theorem}[Federer \cite{Fe}, 4.1.20] Let $T$ be a flat current of dimension $m>0$, then $\mathcal{H}^m(\supp{T})=0$ implies $T\equiv 0$. 
\end{theorem}
\begin{proof} We follow the ideas in \cite{HL2} and \cite{HL3}. The main point is to turn the identity above into an equality of kernels (currents in $P\times P$)   which are represented by locally integrally  flat currents. The secret lies in the formula:
\begin{equation}\label{eq4} \phi_t^*(\cdot)=(\pi\circ \pi_2)_*(\pi_1^*(\cdot)\wedge \Imag \xi_t),
\end{equation}
where $\xi_t:=(\phi_t,\phi_0)$ is the flow-graph of $\phi_0$ at time $t$  and $\pi_{1,2}$ are the projections onto the first and the second factor in $P\times P$. Relation (\ref{eq4}) is just a rewriting in terms of currents of the change of variables formula:

\begin{equation}\label{eq5s5} \int_{B}\phi_t^*\omega\wedge \eta=\int_{\xi_t(B)}\pi_1^*\omega\wedge (\pi\circ\pi_2)^*\eta,\quad\quad \forall \omega\in\Omega^k(P), \forall\eta\in\Omega^{n-k}_c(B).
\end{equation}

The current $\Imag \xi_t$ in $P\times P$ satisfies the following equality:
\begin{equation}\label{xicur} \Imag\xi_t-\Imag\xi_0=d[\Imag \xi_{[0,t]}].
\end{equation}
Notice that because $\xi_0$ is a proper map we have that $\Imag\xi_0$ is a closed current in $P\times P$ (as long as the manifold $B$ does not have any boundary).  Hence $\xi_t$ is a closed current. Moreover, it follows from  Theorem \ref{gan} that the  current $\Imag \xi_{[0,t]}$ converges when $t\ra \infty$ locally in  the mass norm to $\Imag\xi_{[0,\infty)}$ hence $d[\Imag \xi_{[0,t]}]$ is a locally integrally flat current. We conclude that $\Imag \xi_{\infty}:=\displaystyle\lim_{t\ra \infty}\Imag\xi_t$ exists and is a closed, locally integrally flat current of dimension $n$ in $P\times P$.  

It is not too hard to see (for example by Theorem \ref{notn})  that  $\supp\Imag\xi_{\infty}\subset \overline {\Imag \xi}$ and in fact one has:
\[\supp\Imag\xi_{\infty}\subset\{(p,q)~|~q\in\Imag \phi_0\;\mbox{and} \;p \succ^{(1)} q \}=:C\] where the order relation $p\succ q$ means that there exists a (possibly broken) trajectory going from $q$ to $p$.  For an integer $l$ we use the notation $p\succ^{(l)} q$ to say that there exists a broken trajectory between $q$ and $p$ with at least $l$ critical points in between, counting $p$ and/or $q$ too provided they are critical. 

Due to the transversality of $\phi_0$ with all the stable bundles one has that 
\[C\supset \{(p,q)~|~q=\phi_0(b)\;\mbox{and}\; p\in U_{\phi_{\infty}(b)}\}=\bigcup_{F}U(F)\times_F \phi_0(\phi_0^{-1}(S(F))).\]
Consider $D:=P\times P\setminus \{(p,p)~|~p\in F\cap \phi_0(B)\}\cup\{(p,q)~|~q\in\phi_0(B),\; p\succ^{(2)}q\}$, an open subset of $P\times P$. It follows from the transversality of $\phi_0$ and the Morse-Smale condition on the flow (see Lemma 3.10 in \cite{La}) that the $n$-Hausdorff measure of $C\setminus D$ is zero. Just as in Lemma 2.5 in \cite{HL3} the closure of $\Imag \xi_{[0,\infty)}\cap D$ in $D$ is a manifold with boundary where the  boundary at $\infty$ is the disjoint union of manifolds: $\bigcup_{F}U(F)\times_F \phi_0(\phi_0^{-1}(S(F)))$. 

Let $T:=\sum_FU(F)\times_F \phi_0(\phi_0^{-1}(S(F)))$ be a locally integrally flat current on $P\times P$ of dimension $n$. It is a current to begin with because of Propositions \ref{essprop1} and \ref{essprop2}. It follows from what we said above that $\Imag \xi_{\infty}\bigr|_{D}=T\bigr|_{D}$ and this translates into $\supp( \Imag \xi_{\infty}-T)\subset C\setminus D$.  Now we can apply  Theorem 4.1.20 in \cite{Fe} to conclude that in fact 
\[ \Imag \xi_{\infty}=T.
\]

The final formula is a matter of making explicit the current $(\pi\circ \pi_2)_*(\pi_1^*(\omega)\wedge T)$ by using the diagram:
 \begin{equation}\xymatrix{ & U(F)\times_{F} \phi_0(\phi_0^{-1}(S(F)))\ar[r]\ar[d]  &U(F)\ar[d] \\
\phi_0^{-1}(S(F)) \ar[r]^{\phi_0} \ar @/_1.5 pc/[rr]_{\tau_F}  & \phi_0(\phi_0^{-1}(S(F)))\ar[r]^{\qquad\pi^s_F}   &  F}
\end{equation} 
A standard result (a combination of Fubini with  Prop VIII,  page 301 \cite{GHV}) implies that
\[\int_{U(F)\times_F \phi_0(\phi_0^{-1}(S(F)))}\pi_1^*\omega\wedge\pi_2^*(\pi^*\eta)=\int_{\phi_0(\phi_0^{-1}(S(F)))}(\pi_F^s)^*\left(\int_{U(F)/F} \omega\right)\wedge \pi^*\eta \bigr |_{\phi_0(\phi_0^{-1}(S(F)))}.
\]
Finally, $\phi_0:\phi_0^{-1}(S(F))\ra \phi_0(\phi_0^{-1}(S(F)))$ is a diffeomorphism and $\pi\circ \phi_0=\id$.
\end{proof}
\begin{remark} One issue we have not treated above is the orientation. Since the manifold $B$ is oriented one gets that $\xi_0(B)$ is oriented by requiring that $\xi_0$ be orientation preserving. One also has that $U(F)\times_F \phi_0(\phi_0^{-1}(S(F)))$ is naturally oriented as a connected component of the (codimension $1$) boundary at $\infty$ of $\overline{\xi([0,\infty)\times B)}$. One can choose on orientation for a fiber of $U(F)\ra F$ (the fiber is orientable since it is homeomorphic with some $\bR^i$).  One endows $\phi_0(\phi_0^{-1}(S(F)))$ with the complementary orientation in $U(F)\times_F \phi_0(\phi_0^{-1}(S(F)))$ using a certain universal convention say "unstable fiber first".  Now the resulting current $\int_{U(F)/F}\omega\times[\phi_0(\phi_0^{-1}(S(F)))]$ only depends on the orientation of $U(F)\times_F \phi_0(\phi_0^{-1}(S(F)))$ which is induced by the orientation on $B$. It does not depend on the choice of orientation on the fiber of $U(F)\ra F$ since any change here will be offset by a mandatory change on $\phi_0(\phi_0^{-1}(S(F)))$. Finally the orientation on $\phi_0^{-1}(S(F))$ is the one induced by $\phi_0$.
\end{remark}
\begin{remark} We pretty much followed the proof of Theorem 2.3 in \cite{HL2} which is based on Federer Theorem 4.1.20. An often cited result in the works of Harvey, Lawson and Latschev is Federer's Flat Support Theorem, a generalization for flat currents of the well-known Constancy Theorem. The Flat Support Theorem (FST) says that if the support of a $k$-dimensional, \emph{closed}, flat current $T$ is contained in a submanifold $N$ of the same dimension $k$,  then  $T=c[N]$ for some real constant $c$. This theorem is a consequence via localization  of the Constancy Theorem and Federer's characterization of flat currents in section 4.1.15. of his book \cite{Fe}.   It is crucial that one does not need to know a priori that $[N]$ is also a closed current in order to apply this result. An alternative proof to Theorem \ref{refhomf} can be given using FST,  by noticing that the support of the current of interest ($\displaystyle\lim_{t\ra \infty} \Imag\xi_t$) is contained in  the topological closure of $N:=\sqcup_F U(F)\times_F \phi_0(\phi_0^{-1}(S(F)))$, which is itself a rectifiable current and for which FST applies. We preferred the "cleaner" 4.1.20 instead even if that means a bit more work. 

Notice however that whatever method one chooses to employ one needs to know a priori the validity of Propositions \ref{essprop1} and \ref{essprop2} in order to make sense of the current at infinity.
\end{remark}

 In the rest of the article we will do several examples where Theorem \ref{refhomf} applies. Let us see an easy example where non-compactness is present.
\begin{example} If $E\ra B$ is a Riemannian vector bundle with a section $s$ and $f(x,v)=-|v|^2$ and let $\Phi$ denote the flow it induces.  Then for every form $\omega\in \Omega^*(E)$ we get that
\[ \lim_{t\ra \infty} \Phi_t^*\omega=\iota^*\omega, 
\]
where $\iota:B\ra E$ is the zero section inclusion. The only stable manifold in this case is $E$ itself and obviously every section is transversal.
\end{example}

Putting together Theorem \ref{refhomf} with Corollary \ref{cltran} we get:
\begin{corollary}\label{impcor} Let $\omega\in\Omega^k(P)$ be a {closed} form, $X\in\Gamma(VP)$ a vertical, locally Morse-Bott-Smale vector field and $\phi_0:B\ra P$ a smooth section of a fiber bundle with {compact fiber} which is normal with respect to $X$. Then the following generalized transgression formula holds:
\[ \sum_{\codim S(F)\leq k}\Res^u_F(\omega)[\phi_0^{-1}(S(F))]-\sum_{\codim U(F)\leq k}\Res^s_F(\omega)[\phi_0^{-1}(U(F))]=d\mathscr{T},
\]
where $\mathscr{T}$ is a current with $L^1_{\loc}$ coefficients.
\end{corollary}

\begin{example} \label{exlahl}  We do now a  particular important case of Theorem \ref{refhomf} which is also a \emph{universal} case, in which the role of the  section (here a tautological one) goes almost unnoticed.   Let $\pi :P\ra B$ be a fiber bundle and $X$ a vertical vector field which is locally Morse-Bott-Smale, the gradient of an upper bounded function in each fiber. Let $\omega\in\Omega^*(P)$ be a form on $P$. Let $\pi_2:P\times_{B}P\ra P$ be the first projection of the fiber product built out of $\pi$. It comes with a tautological section:
\[ s:P\ra P\times_BP, \quad\quad s(p)=(p,p).
\]
Now, $V\pi_2:=\Ker d\pi_2$ is naturally isomorphic to $\pi_1^*(VP)$ where $\pi_1$ is the other projection. Hence the vector field $X$ gets lifted to a section of $V\pi_2\ra P\times_B P$. Obviously, this vertical vector field on $P\times_BP$ is locally Morse-Bott-Smale just like its father. The tautological section is transversal to all stable or unstable manifolds of this flow as the latter are of the type $N\times_{B}P$ where $N$ is a stable/unstable manifold of $X$. In fact, the flow $\tilde{\Phi}$ on $P\times_BP$ is related to the flow $\Phi$ on $P$ by the simple relation:
\[ \tilde{\Phi}(t,p,q)=(\Phi(t,p),q),
\]
while  $\phi(t,p):=\tilde{\Phi}(t, s(p))=(\Phi(t,p),p)$. Notice that
\[ \pi_1\circ \phi=\Phi,
\]
which motivates the application of Theorem \ref{refhomf} to the form $\pi_1^*\omega$. 
\end{example}
 For computing the residues one uses the straightforward relation:
\[ \int_{U(\tilde{F})/\tilde{F}}\pi_1^*\omega= \pi_1^*\int_{U(F)/F}\omega,
\]
where $\tilde{F}$ is a critical manifold in $P\times_BP$. One clearly has $s^{-1}(S(\tilde{F}))=S(F)$. The following beautiful formula  recovers Latschev's Theorem 4.1 in \cite{La} and Harvey and Lawson Theorem  3.3 in \cite{HL3}. 
\begin{corollary} \label{corLat}
For every form $\omega\in\Omega^k(P)$ the following relation holds:
\begin{equation}\label{cur20} \lim_{t\ra \infty}\Phi_t^*\omega=\sum_{\codim{S(F)}\leq k}\left(\int_{U(F)/F}\omega\right)[S(F)].
\end{equation}
\end{corollary}
One can see that (\ref{cur20}) is in fact a "translation"  of an equality of kernels on $P$ (by definition these are currents on $P\times_BP$). In order to justify it, one does a bit of yoga with standard operations with currents,  just as in the Appendix of \cite{HL3}  or \cite{La}.
 \[\boxed{ \lim_{t\ra \infty}\Gamma_{\Phi_t}=\sum_{F} [U(F)\times_B S(F)]}
\]

\section{\bf {A non-compact  version}} \label{Secnc}

In applications it is many times useful to have the possibility to apply Theorem \ref{refhomf} even when the action function that gives the gradient in each fiber is not bounded above as in the previous section.  The obvious example that comes in mind is  the flow on a Riemannian vector bundle generated by the gradient of the norm squared - as opposed to the negative of the same function (see  Getzler's theorem  \cite{Ge} in Section \ref{sult}).  We would still like to have a transgression formula. However, one cannot expect  geometric atomicity in this context.

Nevertheless, one might have already noticed that the results of Section \ref{s2} are not the most general possible. In fact, let $P\ra B$ be a fiber bundle, $g$ a Riemannian metric on $P$, $\omega$ a $k$-form on $P$ with $k\leq \dim{B}$, $X$ a vertical vector field and $\phi_0:B\ra P$ a section.
\begin{definition}\label{sattup} The tuple $(g,\omega, X, \phi_0)$ is said to be weakly atomic on the positive semi-axis if for every $b\in B$ there exists a neighborhood $b\in U$ such that
\[ \int_{\phi([0,\infty)\times U\setminus Z)} |\omega|~d\mathcal{H}^{n+1}<\infty,\quad\quad \int_{\phi([0,\infty)\times U\setminus Z)} |d\omega|~d\mathcal{H}^{n+1}<\infty,
\]
where $\mathcal{H}^{n+1}$ is the $(n+1)$-Hausdorff measure on $P$ induced by $g$ and $Z:=\phi^{-1}(X^{-1}(0))$. 

The tuple $(g,\omega, X, \phi_0)$ is said to be strongly atomic on the positive semi-axis if for every $b\in B$ there exists a neighborhood $b\in U$ such that
\[\int_{\xi([0,\infty)\times U\setminus Z)} |\pi_1^*\omega|~d\mathcal{H}^{n+1}<\infty,\qquad \int_{\xi([0,\infty)\times U\setminus Z)} |\pi_1^*d\omega|~d\mathcal{H}^{n+1}<\infty,
\]
where $\xi:=(\phi_t,\phi_0)$ is the graph of $\phi$ in $P\times_B P$ and $\pi_1:P\times_B P\ra P$ is the projection onto the first coordinate.
\end{definition} 

\begin{remark} Notice that a tuple $(g,\omega, X, \phi_0)$ such that $\phi_0$ is weakly/strongly geometrically atomic with respect to $X$ and such that  $\omega$ and $d\omega$ are vertically bounded is a weak/strong atomic tuple. Hence geometric atomicity is the lighter version of the above definition.
\end{remark}

Recall that on an oriented manifold $M$ a $k$-form $\tilde\omega$   defines a current by setting:
\[ \Omega^{n-k}\ni\eta\ra\int_M \tilde{\omega}\wedge \eta.
\]
If $M$ has a Riemannian metric then one has an alternative definition:
\[ \eta\ra \int_M\tilde{\omega}\wedge \eta (\overrightarrow{M})~d\mathcal{L},
\]
where $\overrightarrow{M}$ is  the unit, positive orientation multi-vector and $\mathcal{L}$ is the measure induced by the volume form. This alternative definition is useful if one wants to extend the domain of definition of the current represented by $\tilde{\omega}$ from forms with compact support to more general forms. For example if $|\tilde{\omega}|$  is an integrable function with respect to $\mathcal{L}$, then due to the inequality:
\[ |\tilde{\omega}\wedge \eta|\leq C_{k,n}|\tilde{\omega}|\cdot |\eta|,
\] 
one can extend the domain of $\tilde{\omega}$ to bounded (in the norm induced by the metric) forms $\eta$. 

We can now use one of the maps $\phi$ or $\xi$ to induce a metric onto $[0,\infty)\times B\setminus Z$.  Definition \ref{sattup} insures that for an atomic tuple the following integral makes sense for any form $\eta$ on $B$ with compact support
\[ \int_{[0,\infty)\times U\setminus Z}\phi^*\omega\wedge p_2^*\eta,
\] 
where $p_2:[0,\infty)\times B\ra B$ is the projection onto the second factor since $p_2^*\eta$ will be bounded.  

\begin{remark}
The conclusions of Proposition \ref{p1} and Theorem \ref{t1} are valid for weak/strong atomic tuples. 
\end{remark}

We now go back to the locally constant Morse-Bott-Smale vertical gradients. The secret lies again in the magic  formulas (\ref{eq4}) - (\ref{xicur}). 

 The currents $\Imag \xi_{[0,t]}$ will still converge \emph{locally} in the mass norm  even if their support will be non-compact.  Once we have this,  we can  use the same techniques to prove the convergence of $\xi_t$  to a locally flat current since Federer's Theorem 4.1.20 applies to locally flat currents as well.  Therefore the proof of  Theorem \ref{refhomf} goes through to give us:
\begin{equation}\label{eq1s6}  \lim_{t\ra \infty} \xi_t= \bigcup_{F}U(F)\times_F \phi_0(\phi_0^{-1}(S(F)))=:\xi_{\infty}
\end{equation}

 Finally, we need an \emph{extension of the domain} of definition of the currents $\xi_t$ with $t\in[0,\infty]$ to accommodate forms of type $\pi_1^*\omega$ as in $(\ref{eq5s5})$ which will surely have non-compact support. This is where the notion of atomic tuple enters. 
  We now describe the data.

Let $P\ra B$ be a fiber bundle with a  Riemannian metric $g$ such that the restriction of $g$ to $VP$ is complete in every fiber.   Suppose we have a  vertical gradient vector field $\nabla f$ induced by a smooth function $f:P\ra B$. Assume that $\nabla f$ is locally horizontally constant as in Section \ref{section 3} and that it is Morse-Bott-Smale in every fiber. 

Moreover, assume that  $f$ is \emph{proper}  with \emph{a finite number of critical manifolds}.  Assume also that  $\nabla^Vf$ is a complete vector field in every fiber.
\begin{remark}\footnote{We owe this observation to Luciano Mari.} On a complete manifold $M$, a sufficient condition for a vector field to be complete is $|X|_p\leq C\dist(p,o),\;\forall p \in M$ where $o$ is a fixed point on $M$. In fact, one cannot hope to do much better than that. If the growth  of $|X|$ at $\infty$ is say of type $p\ra \dist(p,o)^{1+\epsilon}$  then solutions which blow-up in finite time show up. For example,  the function on the real line $f(x)=1/3x^3+x$ is proper and Morse (no critical points) with $t\ra \tan(t)$ a solution of $x'=\nabla_xf$. 
\end{remark}

Let $\phi_0:B\ra P$ be a section which is transversal to all the stable manifolds of $f$ and let $\phi:\bR_{\geq 0}\times B\ra P$ be  $\phi_0$ flown to $\infty$. Assume $\omega \in\Omega^k(P)$ is a smooth form such that $(g,\nabla f, \omega, \phi_0)$ form a \emph{strong atomic tuple}. We will also assume that for every critical manifold $F$ the following holds: there exists a $C>0$ such that
\begin{equation}\label{eq2snv} \int_{U(q)}|\omega|~d\mathcal{H}^p<C, \qquad \forall  q\in \phi_{\infty}(\phi_0^{-1}(S(F)))\subset F
\end{equation}
where $p=\codim {S(F)}$ and the Hausdorff measure $\mathcal{H}^p$ induced by the metric $g$ is restricted to the fiber $U(q)$ of the bundle $U(F)/F$. Obviously, a form with compact support satisfies such a requirement given a Morse-Bott-Smale flow. We need (\ref{eq2snv}) to make sense of the residues in Theorem \ref{refhomf}.

 Recall now that $\xi_0:B\ra P\times_BP$ is defined by $\xi_0:=(\phi_0,\phi_0)$ and we let $\xi$ be the flow of $\xi_0$ via the gradient of the function $\tilde{f}:P\times_B P\ra\bR$ defined by $\tilde{f}(x,y):=f(x)$. Then $\xi_{[0,t]}\subset P\times_BP$ is a rectifiable current that converges locally in the mass norm when $t\ra \infty$. Recall that this means the following: there exists a current $\xi_{[0,\infty]}$ such that $\xi_{[0,t]}\ctr \gamma\ra \xi_{[0,\infty]}\ctr \gamma$  in the mass norm for every smooth function with compact support $\gamma$.  This implies the convergence of $\xi_t$ locally in the flat norm to a current $\xi_{\infty}$. The same arguments as in the proof of Theorem \ref{refhomf} prove (\ref{eq1s6}).

Hence we can write:
\begin{equation}\label{eqxi} (d\xi_{[0,\infty]})(\theta)=\xi_{\infty}(\theta)-\xi_0(\theta), \quad\quad \forall \theta\in\Omega_{\cpt}^{n}(P\times_B P).
\end{equation}

It does not harm in assuming that $B$ is compact in what follows and since the discussion is local in $B$ we can safely suppose that $P$ is diffeomorphic with $B\times F$. By the assumption of vertical completeness for the Riemannian metric on $P$,  there exists an exhaustion with compacts $(K_i)_{i\in \bN}$ of $P$ and cut-off functions $\gamma_i:P\ra \bR$ with $0\leq \gamma_i\leq 1$ such that:
\[ \gamma_i\equiv 1 \quad\mbox{on } \; K_i,\quad\quad |d\gamma_i|\leq 1\; \mbox{on}\; P\qquad\mbox{and} \qquad \supp \gamma_i\subset K_{i+1}.
\]
\begin{lemma} \label{le1s6} Suppose $\theta\in \Omega^{n+1}(P\times_BP)$ is a form such that 
\[ \int_{\xi{[0,\infty)}}|\theta|~d\mathcal{H}^{n+1}<\infty\quad\mbox{and}\quad\int_{\xi[0,\infty)}|d\theta|~d\mathcal{H}^{n+1}<\infty.
\]
Then $\displaystyle \lim_{i\ra \infty} d\xi_{[0,\infty]}(\gamma_i\theta)=d\xi_{[0,\infty]}(\theta):=\xi_{[0,\infty)}(d\theta)$.
\end{lemma}
\begin{proof} We have:
\[ d\xi_{[0,\infty)}((1-\gamma_i)\theta)=\xi_{[0,\infty)}(-d\gamma_i \wedge \theta)+\xi_{[0,\infty)}((1-\gamma_i)\wedge d\theta).\]
On the other hand for every integrable $n$-form $\alpha$, one has:
\[ \left|\int_{\xi([0,\infty])}\alpha\right|\leq \int_{\xi[0,\infty]}|\alpha|~d\mathcal{H}^n.
\]
 Lebesgue dominated convergence Theorem gives us the claim due to the uniform bounds on $|d\gamma_i|$ and $\gamma_i$.
\end{proof}
We can use the above result for the form $\theta:=\pi_1^*\omega\wedge \pi_2^*\pi^*\eta$ since $\eta$ coming from $B$ is uniformly bounded. Similarly one has due to \ref{eq2snv}:
\begin{lemma}\label{le2s6} Suppose $\omega$ is a form on $P$ that satisfies (\ref{eq1s6}). Then $\displaystyle\lim_{i\ra \infty} \int_{U(F)/F}\omega\gamma_i$ exists and is a smooth form on $F$. It is denoted by $\int_{U(F)/F}\omega.$
\end{lemma}
Putting together (\ref{eqxi}), Lemma \ref{le1s6} and Lemma \ref{le2s6} and Corollary \ref{cltran} we get the non-compact extension we were looking for.
\begin{theorem}\label{nncext} Let $(g,\omega, \nabla^V f, \phi_0)$ be a strongly atomic tuple on the positive semi-axis in the fiber bundle $P\ra B$. Suppose  that:
\begin{itemize}
\item[(a)] with respect to $g$, all fibers of $P\ra B$ are complete;
\item[(b)] $\nabla ^V f$ is a complete, locally constant Morse-Bott-Smale vertical vector field which is proper with a finite number of critical manifolds;
\item[(c)]  $\omega$ satisfies the integrability condition (\ref{eq2snv}); 
\item[(d)] $\phi_0$ is transversal to all the stable manifolds.
\end{itemize}
Then the following homotopy formula holds:
\[  \sum_{\codim{S(F)}\leq\deg{\omega}} \Res_F^u(\omega)[\phi_0^{-1}(S(F))]-\phi_0^*\omega=(-1)^{|\omega|}d[\mathscr{T}_{\infty}(\omega)]+\mathscr{T}_{\infty}(d\omega) ,
\]
where $\mathscr{T}_{\infty}(\omega)$ and $\mathscr{T}_{\infty}(d\omega)$ are currents with $L^1_{\loc}$ coefficients and $\Res_F^u=\tau_F^*\left(\displaystyle\int_{U(F)/F}\omega\right)$.
\end{theorem}

\section{\bf {The top Chern class}} \label{tCc}
The first application we  discuss is "the" standard example. It also appears in \cite{HL1} within the general framework of  the theory of singular connections and also in \cite{HL2} as an application of the general theory. We  choose to present all the details here, if not  for anything else, then for didactical reasons, but also because we use slightly different conventions than \cite{HL2}.  This non-trivial case is a beautiful illustration of how one can keep the computations   to an absolute minimum. In fact, one only needs one universal computation for a universal constant. 

Let $E\ra B$ be a complex vector bundle over a compact manifold $B$ of rank $n$ endowed with a connection $\nabla$ and a section $s:B\ra E$ transversal to the $0$ section. We assume that $E$ comes with a hermitian metric but we do not assume compatibility of $\nabla$ with the metric. Let $\pi:P:=\bP(E\oplus \bC)\ra B$ be the projective space fiber bundle over $B$ obtained from $E$. Clearly $s$ can be seen as a section of $\bP(E\oplus \bC)$ by considering the graph of $s$. 

Using the hermitian metric, on each complex projective space $\bP(E\oplus \bC)_b$ one has a Morse-Bott function given by
\[ f_b(L)=\Real\Tr{\epsilon P_L},
\]
where $P_L\in \End(E_b\oplus \bC)$ is the orthogonal projection on $L$ and $\epsilon=\left(\begin{array}{ccc} 1&0 \\
  0& -1\end{array}\right)$ is the  reflection that induces the $\bZ_2$-grading of $E_b\oplus \bC$.
 
    It is quite an easy exercise to show that $[0:1]$ is a unique point of minimum  and the manifold $\bP_{\infty}:=\{[0,v]~|~v\in E\}$ is the critical manifold of maximum points and there are no other critical points for $f$ (see for example Exercise 5.23 in \cite{Ni2}). Moreover $f$ is Morse-Bott-Smale. As a matter of fact, the gradient flow of $f$ is the compactification of the linear flow $v\ra sv$ in $E$ which at $s=0$ has a critical point in $[1:0]$ and at $s=\infty$ has a critical manifold in $\bP_{\infty}$.
  
 The unstable manifold of $[1:0]$ is $\Hom(\bC,E)_b\subset \bP(E\oplus \bC)_b$ while the stable manifold is $S([1:0])=\{[1:0]\}$. Similarly $U(\bP_{\infty})=\bP_{\infty}$ and $S(\bP_{\infty})=\bP(E\oplus \bC)_b\setminus\{ [1:0]\}$.  
  
  One has a global such function $f$ and $X$ will be the vertical gradient vector field it determines. We denote by the same symbols as above the critical/stable/unstable manifolds of $X$ in $P$, hoping this will not cause confusion.
  
  Let $\tau\ra P$ be the tautological line bundle over $P$.  The form $\omega$ on $P$ is the $n$-th Chern form of the connection $\tina$ on $\tau^{\perp}$ obtained by orthogonally projecting $d\oplus \pi^* \nabla$ onto $\tau^{\perp}$. Hence
  \[ \omega=c_n(F(\tina))=\left(\frac{i}{2\pi}\right)^n\det(F(\tina)).
  \]
  
  A crucial observation is that 
  \begin{equation}\label{crobs}\left(\tau^{\perp}\bigr|_{[1:0]},\tina\bigr|_{[1:0]}\right)\simeq (E,\nabla)
  \end{equation} canonically, as bundles with connections. 
  
  We apply Corollary \ref{impcor} since all the transversality conditions are met for $s$. The current at $-\infty$ is of the form 
  \[  \omega_1[s^{-1}(\Hom(\bC,E))]+\omega_2[s^{-1}(P_{\infty})],
  \]
  where $\omega_2$ does not matter since $s^{-1}(P_{\infty})=\emptyset$. On the other hand, $\omega_1$ is the pull-back via $\id_B=\tau_{[1:0]}:s^{-1}(\Hom(\bC,E))\ra B$ of $c_n(F(\tina\bigr|_{[1:0]}))$. By (\ref{crobs}) we get $\omega_1=c_n(F(\nabla))$. 
  
  The current at $\infty$ is of the form
  \[ \eta_1[s^{-1}(0)]+\eta_2[s^{-1}(E\setminus \{0\})]
  \]
  Here $\eta_2$ is the restriction of $c_n(F(\tina))$ to $\bP_{\infty}$, pulled-back via the application $\tau$ to $E\setminus\{0\}$. Now over $\bP_{\infty}$ we have the following splitting of $(\tau^{\perp}\bigr|_{\bP_{\infty}},\tina\bigr|_{\bP_{\infty}})$:
  \[ (\tau^{\perp}\bigr|_{\bP_{\infty}},\tina\bigr|_{\bP_{\infty}})=(\tilde{\tau}^{\perp}\oplus \bC,P_{\tilde{\tau}^{\perp}}(\pi_2^*\nabla)\oplus d)
  \]
  where $\tilde{\tau}^{\perp}\ra P(E)$ is the orthogonal tautological bundle and $\pi_2:\bP(E)\ra B$ is the natural fiber bundle projection. The upshot is that there exists a non-vanishing parallel section for $\tina\bigr|_{\bP_{\infty}}$ and therefore $c_n(F(\tina\bigr|_{\bP_{\infty}}))=0$ and so $\eta_2=0$.
  
  Finally, $\eta_1$ is the restriction to $s^{-1}(0)$ of the fiber integral 
  \[\displaystyle\int_{\Hom(\bC,E)/[1:0]} c_n(F(\tina))\in C^{\infty}(B).\] Since the fiber integral of a closed form is a closed form we get that the result of integration is a constant. So it suffices to do the computations in one fiber. Since $\Hom(\bC,E_b)$ has full measure in $\bP(E_b\oplus \bC)$ we can substitute this later space in the integral in each fiber. Now, $\tina\bigr|_{\bP(E_b\oplus\bC)}$ is the projection to $\tau^{\perp}_b$ of the trivial connection on $E_b\oplus \bC$, the unique connection compatible with the metric and the holomorphic structure, generally called the Chern connection.  In order to conclude this example we have to perform a \emph{universal computation} on a complex projective space. 
  
  We restrict attention to a fiber  but use the same notation without the underscript $b$ . The next result is standard but we include it for completeness sake.
  
The Chern connection on $\tau^{\perp}$ is associated to the left-invariant Maurer-Cartan ($U(n)$-principal bundle) connection on the Stiefel manifold of orthonormal $n$-frames $U(n+1)/U(1)$ over the complex projective space $\bC\bP^n$ via the canonical representation $U(n)\times\bC^n\ra \bC^n$ which induces $\tau^{\perp}$ to begin with. As such, the characteristic forms it determines are invariant under the action of $U(n+1)$ on  the complex projective space.

On the other hand, it is quite straightforward  to calculate  $F(\tina)$  at the point $[0:1]$. We include the proof because we could not find a reference.
\begin{lemma}\label{ftina} The vector spaces $T_{[0:1]} \bP(E\oplus \bC)$ and $\tau^{\perp}\bigr|_{[0:1]}$ can both be canonically  identified with $E$.  Then the curvature of the Chern connection at $[0:1]$ is  the following map from $\Lambda^2E$ into $\mathfrak{u}(E)$, the space of skew-adjoint operators on $E$:
\[ u\wedge w\ra \{v\ra \langle v,w\rangle u-\langle v,u\rangle w\}.
\]which corresponds to the matrix of $2$-forms:
\[ \left(\begin{array}{ccc} \ldots & \ldots &\ldots\\
 \ldots & dz_i\wedge d\bar{z}_j &\ldots \\
 \ldots &\ldots &\ldots \end{array} \right),
\]
once one chooses a basis on $E$.
\end{lemma}
\begin{proof} The standard chart centered at $[0:1]$ of the projective space is the graph map defined on $E\simeq \Hom(\bC,E)\ra \bP(E\oplus \bC)$. In this chart, $\tau^{\perp}$ at a point $v\in E$ consists of the subspace of $E\oplus \bC$ of the form:
 \begin{equation}\label{vww}\{(w, -v^*(w))~|~w\in E\},\end{equation}
 where $v^*\in E^*$ plays the role of the adjoint of $v\in E$ (seen as a linear map $\bC\ra E$). Of course, $v^*(w)=\langle w, v\rangle$ via  metric duality. 
  
  We will denote by $s:E\ra \tau^{\perp}\subset  E\oplus \bC$ a section and by $\tilde{s}=\pi_2\circ s$ the projection onto the $E$-factor which completely determines  $s$.
 
 The projection of $d_vs(w)=\frac{\partial s}{\partial w}(v)=(d_v\tilde{s}(w),-[w^*(\tilde{s}(v))+v^*(d_v\tilde s(w))])$ onto $\tau^{\perp}$ is:
 \[ v\ra (d_v\tilde{s}(w),-v^*(d_v\tilde{s}(w)))+\frac{w^*(\tilde{s}(v))}{v^*(v)+1}(v,-v^*(v)).
 \]
 We differentiate again in the $u$-direction the above function of $v$ and evaluate at $v=0$. The derivative of the first of the two terms in the $u$-direction at $0$ is:
 \begin{equation}\label{vww1} \left( \frac{\partial^2\tilde{s}}{\partial u\partial w}(0), -u^*\left(\frac{\partial \tilde{s}}{\partial w}(0)\right)\right).
 \end{equation}
The derivative in the $u$-direction of the second term is simply:
 \begin{equation}\label{vww2} w^*(\tilde{s}(0))(u,0).
 \end{equation}
 Summing up (\ref{vww1}) and $(\ref{vww2})$ we get a vector whose first component is
  \[\frac{\partial^2\tilde{s}}{\partial u\partial w}(0)+w^*(\tilde{s}(0))u.\] The second component does not matter since it vanishes when we project to $\tau^{\perp}\bigr|_{[0:1]}$.  The last thing left to do is skew-symmetrizing:
 \[ F(\tina)s(0)= \tina_u\nabla_ws-\tina_w\tina_u s=w^*(\tilde{s}(0))u-u^*(\tilde{s}(0))w=\langle s(0),w\rangle u-\langle s(0),u\rangle w,
 \]
where $s(0)=\tilde{s}(0)$ holds because of the identification of $\tau^{\perp}\bigr|_{[0:1]}$ with $E\oplus 0$.
 
Finally, we look at the $a_{ij}$-coefficient of the matrix determined by $u\wedge w$. This is
\[ \langle F(\tina){(u\wedge w)}e_j ,e_i\rangle=\langle e_j,w\rangle\langle u,e_i\rangle-\langle e_j,u\rangle\langle w,e_i\rangle=u_i\bar{w}_j-\bar{u}_jw_i=dz_i\wedge d\bar{z}_j(u\wedge w).\]
\end{proof}

We deduce that at the  point $[0:1]$ one has:
  \[ c_n(F(\tina))=\left(\frac{i}{2\pi}\right)^n\det [dz_i\wedge d\bar{z}_j]_{\substack {1\leq i\leq n\\1\leq j\leq n } }=n!\left(\frac{i}{2\pi}\right)^n\prod_{i=1}^n dz_i\wedge d\bar{z}_i
  \]
  where $\{dz_i,d\bar{z}_{j}\}$ is a basis of $1$-forms in the standard chart centered at $[0:1]$. 
  
   On the other hand, it is standard (a consequence of Wirtinger Theorem, for example) that 
  \[ \int_{\bP(E\oplus \bC)}\eta^{\wedge n}=\deg (\bP(E\oplus \bC))=1,
  \]
  where $\eta$ is the $U(n+1)$-invariant, canonical symplectic form on $\bP(E\oplus \bC)$ that has the expression  $\eta_0:=\frac{i}{2\pi}\sum_i dz_i\wedge d\overline{z}_i$ at $[0:1]$. Notice that $\eta_0^{\wedge n}=c_n(F(\tina))_{[0:1]}$. It follows that the integral we are after is constantly equal to $1$. 
  
  Therefore Corollary \ref{impcor} becomes:
  \[\boxed{ c_n(F(\nabla))-[s^{-1}(0)]=d\mathscr{T}.}
  \]
  \begin{remark} The convention we used here was to compactify the flow $t\ra \tilde\Gamma_{tv}\in \bP(E\oplus\bC)$ where $\tilde{\Gamma}_v:=\{(\lambda v,\lambda)\}\subset E\oplus\bC$ is what we called the switched graph of $v\in E$. This terminology/convention only has relevance in the infinite dimensional case where $v$ represents a Fredholm operator. Nevertheless we chose to stick to it. In contrast, Harvey and Lawson compactify $t\ra \{(v\lambda,t\lambda)~|~\lambda \in \bC\}$.  
    \end{remark}

\section{\bf {The Chern-Gauss-Bonnet Theorem reloaded}}

One finds again at least two different versions of the \emph{local} Chern-Gauss-Bonnet theorem in the work of Harvey and Lawson's work \cite{HL1} and \cite{HZ}. We give yet another one, mainly because of its simplicity and geometric appeal and because it does not require the construction of a Thom form with compact support a priori, which most of the proofs do.

Let $E\ra B$ be an oriented real vector bundle of rank $n$ over a compact manifold $B$. Suppose $E$ is endowed with a Riemannian metric. Let $\nabla$ be a connection on $E$, not necessarily compatible with the metric.  Let $P:=S(\bR\oplus E)\ra B$ be the sphere bundle over $B$. We will regard $E$ embedded in $S( \bR\oplus E)$ via the stereographic map.  Fiberwise, the (positively oriented) stereographic inclusion is given by:
\[ \eSt(v)=\frac{1}{1+|v|^2}(1-|v|^2,2v),\quad\quad v\in E.
\]
 We will  use the symbol $[0]$ for the zero  section in $E$ and for its image $(1,0)$ in $S(\bR\oplus E)$, while $[\infty]$ will represent  the section  $(-1,0)$.  We will also use the same letter $s$ for a section of $E$ as for its image via the stereographic map.

 On $S(\bR\oplus E)$  consider the flow given by minus the "height" function $f(t,v)=-t$. We have $S([\infty])=P\setminus [0]$, $ U([\infty])=[\infty]$, $S([0])=[0]$ and $U(0)=P\setminus [\infty]$.
 
  Notice that $\bR\oplus \pi^*E\ra S(\bR\oplus E)$ has a  line subbundle $\nu$ induced by the tautological section. This tautological section is nothing but the unit normal vector field along each sphere, along each fiber. 
 
  The form $\omega\in\Omega^*(P)$ to be flown is:
  \[ \omega= \frac{1}{(2\pi)^{n/2}}\Pf(\tina),
  \] the Pfaffian of the connection $\tina$ which results from projecting $d\oplus \pi^*\nabla$ (which acts on $\bR\oplus \pi^*E$) onto $\nu^{\perp}$. 
 
 If $s:B\ra E$ is transversal to the zero section then it obviously stays transversal to $[0]$ inside $S (\bR\oplus E)$. 
 
 Corollary \ref{impcor} gives us the following:
 \[  \omega_1 [s^{-1}(P\setminus [\infty])]-(\eta_1 [s^{-1}(0)] +\eta_2[s^{-1}(P\setminus [0])])=d\mathscr{T}.
 \]
where the residues are as follows:
\begin{itemize}
\item $\omega_1$ is the restriction of $\omega$ to $[0]$ and this is exactly $\frac{1}{(2\pi)^{n/2}}\Pf(\nabla)$, since one has a canonical isomorphism of bundles with connection $(\tau^{\perp}\bigr|_{[0]}, \tina\bigr|_{[0]})\simeq (E,\nabla)$;
\item $\eta_2$ is the restriction of $\omega$ to $[\infty]$, which is further pulled-back to $s^{-1}(P\setminus [0])$ via the flow map;
\item $\eta_1$ is the fiber integral $\displaystyle\int_{P/B}\omega$ (restricted to $s^{-1}(0)$). 
\end{itemize}
Notice that due to orientation reasons
\begin{equation}\label{etaome} \eta_2(b)=-\omega_1(b)\quad\quad\forall b\in s^{-1}(P\setminus ([0]\cup [\infty])),
\end{equation}
Indeed, the tangent space to the north pole of the unit sphere (with the standard orientation as the boundary of the unit ball) has the opposite orientation as the tangent space to the south pole, when they both get identified in the canonical way with the horizontal space passing through the origin. Therefore the natural identification north-south reverses orientation and taking into account the symmetry of the $\omega$ one gets the claim (\ref{etaome}). 

We are left with computing $\eta_1=\displaystyle\int_{P/B} \omega$. Now $\omega$ restricted to each fiber is the normalized Pfaffian of the Levi-Civita connection on the unit sphere with the \emph{round metric}. This is the \emph{universal computation} one has to perform in this case. Due to rotational symmetry this Pfaffian is a certain multiple of the standard volume form of the sphere, multiple which can be computed at any favorite point by using a "real version" of Lemma \ref{ftina}. The computations are similar and the result well-known (see for example \cite{Sp}, Lemma 2.18, page 291)
\begin{lemma}  The tangent space of the sphere $S(\bR\oplus E_b)$ at $(1,0,\ldots ,0)$  can be canonically identified with $E_b$ via $(0,v)\ra v$ and this identification  is orientation preserving. The curvature of the Levi-Civita connection on the standard unit sphere is the following $2$-form with values in $\mathfrak{so}(E_b)$:
\[ F(\nabla)(X\wedge Y)Z=\langle Y,Z \rangle X-\langle X,Z\rangle Y.
\]
\end{lemma}

At this point we can stop pretending we do not know the answer.
\[ \int_{P/B} \omega \equiv\chi(S^n)=1+(-1)^n.
\]
So in the even rank case one gets the following refinement of the Chern-Gauss-Bonnet:
\[ \boxed{2 \frac{1}{(2\pi)^{n/2}}\Pf(\nabla)-2[s^{-1}(0)]=d\mathscr{T}}.
\]

\begin{remark} With hindsight, it is natural to consider in each fiber the space $S(\bR\oplus E_b)$ in view of the previous section. This is the projective space of \emph{oriented} lines inside $\bR\oplus E_b$. The other option would be  to work with twisted differential forms on $\bP(\bR\oplus E_b)$ with fibers which are not oriented and that is the context in which Harvey and Lawson develop their results. 
\end{remark}

\section{\bf{Odd Chern classes and the Maslov spark}} \label{occMs}
The odd Chern classes are characteristic classes associated with (pointwise) invertible endomorphisms $U:E\ra E$ of complex vector bundles and they live in the cohomology groups of odd degree of the base manifold $B$. It is standard (see \cite{Ka}) that such  an invertible endomorphism determines an odd, complex $K$-theory class $[U]\in K^{-1}(B)$. Applying the odd Chern character to $[U]$ one gets a non homogeneous class of odd degree with rational coefficients. In fact one can write:
\[ \ch[U]=\sum_{k=1}^{\rank {E}} \frac{(-1)^{k-1}}{(k-1)!} c_{k-1/2}([U]),
\]
with  $c_{k-1/2}([U])\in H^{2k-1}(B;\mathbb{Z})$ the classes which will be the focus of our attention in this section. If $E$ is the trivial bundle then $c_{k-1/2}$ are  pull-backs via $U$ of the generators of the cohomology ring of invertible matrices.   This group is homotopic with the group of unitary matrices and it is convenient to consider classes of unitary endomorphisms to be representatives of odd $K$-theory.

It is legitimate to ask whether there exists an odd counterpart to the classical Chern-Weil theory. In other words, given a triple $(E,U,\nabla)$ with $E$ a hermitian vector bundle, $U$ a unitary endomorphism and $\nabla$ a connection on $E$, can one manufacture a closed form that represents $c_{2k-1}([U])$ in the de Rham cohomology? The answer is: it depends. The situation is as follows:
\begin{itemize}
\item[(a)] If $(E,\nabla)$ is a flat bundle with a flat connection than the answer is yes. The forms
\[ \gamma_k\tr \wedge^{2k-1}U^{-1}(\nabla U)
\]
where $\gamma_k$ are some universal constants (to be discussed below) are closed forms and represent the classes $c_{2k-1}([U])$. 

In general, if $\nabla$ is not flat, the higher degree forms (for $k>1$) built by the same procedure are not necessarily closed. 
\item[(b)] The class  $c_{1/2}([U])\in H^{1}(B;\mathbb{Z})$ can always be represented by $\frac{1}{2\pi i} U^{-1}(\nabla U)$, whether or not $\nabla$ is flat.
\item[(c)] If one is not afraid of superconnections    then Quillen (\cite{Qu}) shows that one can extend the theory of Chern character forms associated to a superconnection built out of $\nabla$ and a self-adjoint operator to include unitary operators. These closed forms represent $\ch[U]$ in cohomology. The nice thing about Quillen's theory is that these forms are in a certain sense universal and there exists a general framework which includes both even and odd $K$-theory (see Section \ref{SCcf} for details).
 \end{itemize}
 
 Despite the many similarities between the odd and even $K$-theory and we have in mind here the fact that they are both classified by certain grassmannians there is one particular reason for which the analogy cannot be pushed too far, at least when the aim is building a Chern-Weil theory. The reason  is the lack of a  representation theorem such as the  Narasimhan-Ramanan in the odd case.  This celebrated result says that every pair  $(E,\nabla)$ of a vector bundle with connection is isomorphic with the pullback of a universal such pair via a classification map. The universal pair is a tautological bundle  with his Chern connection  over a Grassmannian.  In the odd case, the relevant tautological bundle which lives now over a Lagrangian Grassmannian  is a canonically trivializable vector bundle (see \cite{Ci}, Theorem 3.11 (c))  and therefore there cannot exist a theorem that says that \emph{every} triple $(E,U,\nabla)$ comes from a \emph{universal} such triple via pull-back. 
 
 In this section, we will treat completely the case $(b)$ above and also the case $(a)$ when $(E,\nabla)$ is a trivial bundle with trivial connection.

 Let $E\ra B$ be a hermitian vector bundle and let $U:B\ra U(E)$ be a unitary bundle endomorphism. Suppose $E$ is endowed with a connection, not necessarily metric compatible. Let $P=U(E)$ be the fiber bundle of unitary endomorphisms on which we consider the flow given by the vertical gradient of the function:
 \[ f(U)=\Real\Tr{U}.
 \] 
 This flow has a simple explicit description (see \cite{VD}). In a fiber $U(E_b)$ the gradient is given by $U\ra 1-U^2$ and the flow is completely determined:
 \[ \phi(t,U)=\frac{\tanh{t}+U}{1+U\tanh{t}}.
 \]
  One checks that:
  \[ [\lim_{t\ra \infty} \phi_t(U)](v)=\left\{\begin{array}{cc}-v&  v\in\Ker{(1+U)}\\
      v & v\in\Ker{(1+U)}^{\perp} \end{array}\right. \]\[
       [\lim_{t\ra -\infty} \phi_t(U)](v)=\left\{\begin{array}{cc}v&  v\in\Ker{(1-U)}\\
      -v & v\in\Ker{(1-U)}^{\perp} \end{array}\right.
 \]
 The critical manifolds are in bijective correspondence with $\bigcup_{k=0}^n \Gr(k,E_b)$ where each $L\in \Gr(k,E_b)$ determines a unitary reflection:
 \[ U_L=(-\id_{L})\oplus \id_{L^{\perp}}.
 \]
 The stable and unstable manifolds of the critical manifold $\Gr(k,E_b)$ are
 \[ S(\underline {k}):=\{U\in U(E_b)~|~\dim{\Ker{(1+U)}}=k\} \]\[ U(\ul{k}):=\{U\in U(E_b)~|~\dim{\Ker{(1-U)}}=n-k\}.\]
 The set $S(\ul{1})$ is sometimes called the Maslov cycle and we will use a special notation for it: $\mathscr{M}$.
 
 The form of interest here is $\omega=\frac{1}{2\pi i} \tr U^{-1}(\nabla U)$. This form $\omega$ is in fact the pull-back via the section $U$ of a form on $P=U(E)$.  Indeed, one can take $\tilde{U}$ to represent the tautological unitary endomorphism of $\pi^*E\ra U(E)$ and with  $\pi^*\nabla$ on $\pi^*E$ one considers $\tilde{\omega}:=\frac{1}{2\pi i}\tilde{U}(\pi^*\nabla \tilde{U})$.  In any case, we have the following.
 
  \begin{lemma} The form $\frac{1}{2\pi i} \tr U^{-1}(\nabla U)$ is closed.
 \end{lemma}
 \begin{proof} In local coordinates if $\nabla=d+\Theta$ then $(d+\Theta)U=dU+[\Theta,U]$ and therefore:
 \[ \tr U^{-1}(\nabla U)=\tr U^{-1}(dU)+\tr(U^*\Theta U)-\tr \Theta=\tr U^{-1}(dU).
 \]
 Now $\tr U^{-1}(dU)$ is closed as follows from the relations:
 \[ [d,\tr]=0,\quad\quad d(U^{-1}(dU))=-\frac{1}{2}[U^{-1}(dU),U^{-1}(dU)],\quad \mbox{and}\quad \tr[\cdot,\cdot]=0.
 \]
 \end{proof}
 
 If we assume that $U$ is stably normal, meaning that $U$ is transverse to all $S(k)$, then we can apply Theorem \ref{refhomf} to conclude that
 \begin{equation}\label{Maseq}\boxed{\frac{1}{2\pi i} \tr U^{-1}(\nabla U)-[U^{-1}(\mathscr{M})]=d\mathscr{T}.} 
 \end{equation}
 The only explanation missing is the computation of the residue 
 \[\int_{U(\ul{1})/ \Gr(1,n)}\tilde{\omega}.\]
 Now $U(\ul{1}):=\{U~|~\dim{\Ker{(1-U)}}=n-1\}$ is a manifold of real dimension $2n-1$ and as unstable manifold it has a Bott-Samelson resolution in the form of the map:
 \begin{equation}\label{BSres}
  S^1\times \Gr(1,n)\ra U(n), \qquad (\lambda, L)\ra (\lambda \id_{L})\oplus \id_{L^{\perp}}\in U(L\oplus L^{\perp}) \end{equation}
 Notice that the image of this map is exactly $U(\ul{1})\cup \{\id\}$. The fiber of $U(\ul{1})\ra \Gr(1,n)$ can be identified with $S^1\setminus\{1\}$ and it is not hard to see that the form $\frac{1}{2\pi i}\tr U^{-1}(dU)$ on $U(\ul{1})$ (indeed, the trivial connection as one does integration fiber by fiber) pulls back to the angular form $\frac{1}{2\pi i}\lambda^{-1}d\lambda$ on $S^1$ which  integrates to $1$. This finishes the proof of (\ref {Maseq}).  
 
\begin{remark} The equation (\ref{Maseq}) is an example of a spark  equation of degree $0$, in the terminology introduced in \cite{HLZ}. It induces a degree $0$ de Rham-Federer differential character.
 \end{remark}

 \vspace{0.3cm}
 We now turn to $(a)$. As we said for a general hermitian vector bundle $E$ with unitary morphism $U$ and connection $\nabla$ the naive choice of forms $\tr\wedge^{2k-1} U^{-1}(\nabla U)$ does not produce closed forms in general. We will therefore assume that $E$ is the trivial bundle with the trivial connection.  Moreover, we will take $U:\underline{\bC^n}\ra \underline{\bC^n}$ to be a unitary endomorphism of vector bundles over $B$. Let us say that from the point of view of  odd $K$-theory this  is not a dramatic restriction as every class in $K^{-1}(B)$ can be represented by a triple of this type.
 
  It turns out that even if we had a way to produce closed forms then there is no candidate for a good flow. The flow we used above does not serve as the other stable manifolds $S(k)$ with $k\geq 2$ have codimension $k^2$ and therefore their preimages cannot be Poincar\'e duals to forms of degree $2k-1$.

 For all this reasons we restrict our attention to maps $U:B\ra U(n)$ and consider the flow on $U(n)$ induced by a self-adjoint matrix $A$ with distinct eigenvalues $\{\lambda_1<\lambda_2< \ldots<\lambda_n\}$ via the gradient of the map:
 \[  f_A:U(n)\ra \bR,\quad\quad f(U)=\Real\Tr(AU).
 \] 
 Such flows  have been extensively studied in \cite{VD} and \cite{Ni3}. We recall now their explicit form (Proposition 2.1 in \cite{VD}).
 
 \begin{prop} The function $f$ is Morse and its  gradient flow  is determined by:
 \[(t, U_0)\ra (\sinh{(At)}+\cosh{(At)}U_0)(\cosh{(At)}+\sinh{(At)}U_0)^{-1}.
 \]
 \end{prop}
In order to better analyze this flow one fixes an orthonormal basis $\{e_1,\ldots, e_n\}$ of eigenvectors for $A$, $Ae_i=\lambda_i$ and a flag
\[ W_0=\bC^n\supset W_1\supset\ldots\supset W_{n-1}\supset W_{n}=\{0\}
\] 
 such that $W_{k}^{\perp}=\langle e_1,\dots, e_k\rangle$. This choice of the flag is related to the definition of Schubert varieties on the Lagrangian Grassmannian in the infinite dimensional context (\cite{Ci}).
 
 The critical points of this flow are orthogonal reflections $U=-\id_{V}+\id_{V^{\perp}}$, where $V$ is an eigenspace of $A$. One can take obviously take $V=\langle e_{i_1},\ldots, e_{i_{k}}\rangle$, for some ordered set $I=\{i_1<\ldots< i_k\}\subset\{1,2,\ldots,n\}$ and we use notation $U_I$ for this critical point. 
 
 The stable manifold of $U_I$ is  a certain Schubert manifold defined by the following incidence relations where we set $i_0:=0$ and $i_{k+1}=\infty$:
 \[ S(U_I)=\{U~|~\dim{[\Ker{(1+U)}\cap W_m]}=k-p, \quad \forall 0\leq p\leq k,\quad \forall i_p\leq m< i_{p+1}   \}.
 \]
 The incidence relations are saying  that $\dim{\Ker{(1+U)}}=k$ and the numbers $i_1$,\ldots ,$i_k$ record the "nodes" of the flag where the dimension of the intersection $\Ker{(1+U)}\cap W_m$ drops by one. These correspond exactly to the Schubert manifolds considered in \cite{Ci} on the Lagrangian Grassmannian $\Lag(\bC^n\oplus\bC^n)$ via the Arnold diffeomorphism:
 \[ U\ra \Imag\{ v\ra ((1+U)v,-i(1-U)v)\}.
 \]   It is not difficult  to check that $S(U_V)$ has codimension $N_I:=\displaystyle\sum_{i\in I} 2i-1$ in $U(n)$ as proved in \cite{Ci}. Let us emphasize that for $I=\{k\}$, the closure of $S(U_{\{k\}})$ is also the closure of the following incidence manifold:
 \[ Z_{\{k\}}=\{U~|~\dim{[\Ker{(1+U)}\cap W_{k-1}]}=1\}
 \]
 
 Dually one can show that 
 \[ U(U_I)=\{U~|~\dim{[\Ker{(1-U)}\cap W_m]}=n-k-q,\quad\forall 0\leq q\leq n-k,\; \forall j_q\leq m < j_{q+1}\},
 \]
 where $\{j_1<\ldots< j_{n-k}\}=: I^c$ is the complement set of $I$. It has dimension $N_I$.
 
Let $g^{-1}dg$ be the Maurer-Cartan form on $U(n)$. The forms to be flown are:
 \[  c_{k-1/2}:=-\left(\frac{i}{2\pi}\right)^k\frac{[(k-1)!]^2}{(2k-1)!}\tr \wedge^{2k-1}g^{-1}dg,
 \] 
One shows easily that these forms are closed using the relation $d\tr \omega=\tr [d,\omega]$ . Therefore $c_{k-1/2}$ determines a class in $H^{2k-1}(B)$ by pulling it back via the initial map/section $U$. The constants are of course chosen a posteriori after computing the residues. We explain now how this is done.

 The relevant unstable manifold where the integral of $c_{k-1/2}$ is not zero (see (\ref{intck}) and Appendix B) is
 \[U(U_{\{k\}})=\{U~|~ \dim{[\Ker{(1-U)}\cap W_m]}=n-1-m,\quad \forall 0\leq m\leq k-1,\qquad\]
  \[ \qquad \dim{[\Ker{(1-U)}\cap W_m]=n-m,\quad\forall k\leq  m\leq n} \}
  \]
 Notice that if $U\in U(U_{\{k\}})$ then $\Ker{(1-U)}\supset W_k$ and to simplify the computation we mod out $\Ker{(1-U)}$ by considering the analogous  unstable manifold  $\tilde{U}(U_{\{k\}})\subset U(k)$ and the obvious diffeomorphism:
 \[ \tilde{U}(U_{\{k\}})\ra U(U_{\{k\}}),\quad\quad U\ra U\oplus \id_{W_{k}},
 \]
 induced by the natural inclusion $U(k)\hookrightarrow U(n)$ with  the same expression.
 
Notice that the incidence manifold $\tilde{U}(U_{\{k\}})$  is an open dense subset of the manifold
 \[\mathscr{U}_k:=\{U\in U(k)~|~\dim{\Ker{(1-U)}} =k-1\}\]
  because, generically a hyperplane in $\bC^k$ will intersect $W_i$ in dimension $k-1-i$. This manifold already appeared in (\ref{BSres}) with a Bott-Samelson resolution. In fact, topologically, the closure of $\mathscr{U}_{k}$ is the one point compactification of the trivial bundle $\bR\ra \bC\bP^k$ and  it is homeomorphic with $\Sigma\bC\bP^k$.

Now, the Maurer-Cartan form on $U(n)$ pulls back to the Maurer-Cartan form  on $U(k)$. Using the above Bott-Samelson resolution we compute the following integral in Appendix \ref{resU}, where the reader will find also a discussion about orientations conventions.

\begin{equation}\label{intck} \int_{\mathscr{U}_k}c_{k-1/2}=1.
\end{equation}

On the other hand, it is easy to explain why 

\[ \int_{U(U_I)}c_{k-1/2}=0,
\] 
\noindent
for any other $U_I$. Take $I$ such that $U(U_I)$ has dimension $2k-1$ but $U_I\neq U_{\{k\}}$.  Then $\max{\{i\in I\}}\leq k-1$. This implies that $\Ker{(1-U)}\supset W_{k-1}$ for all $U\in U(U_I)$. Indeed, if $\iota=\max{I}$ then for every $m\geq \iota$ we have $m=\iota+s=j_{\iota-l+s}$ for some $s\geq 0$ since there are exactly $\iota-l+s$ numbers smaller than $\iota+p$ which are not in $I$. But this implies that $j_{p}\leq m < j_{p+1}$ with $p=\iota-l+s=m-l$. Therefore, for $U\in U(U_I)$ one has:
 \[ \dim{[\Ker{(1-U)}\cap W_m]}=n-l-p=n-m\]
and this means that $\Ker{(1-U)}\supset W_m$ for all $m\geq \iota$, in particular for $m=k-1$. 

Therefore the manifold $U(U_I)$ is diffeomorphic with the manifold $\tilde{U}(U_I)\subset U(k-1)$  via the canonical  inclusion $U(k-1)\hookrightarrow U(n)$ which takes $U$ to $U\oplus \id_{W_{k-1}}$. But $\tr \wedge^{2k-1}g^{-1}dg$ in $U(k-1)$ is zero.  

The next transgression formula on $U(n)$
\[\boxed{ c_{k-1/2}-S(U_{\{k\}})=d\mathscr{T},}
\]
translates via Theorem \ref{refhomf} when applied to a a section of the trivial bundle $\underline{U(n)}\ra B$ into the following result \cite{Ni3}.
\begin{theorem}[Nicolaescu] Let $B$ be an oriented, compact manifold. For every smooth map $f:B \ra U(n)$ transversal to all the Schubert cells $S(U_I)$, the form $f^*c_{k-1/2}$ and the current $f^{-1}(S(U_{k}))$ are Poincar\'e duals.
\end{theorem} 

 The original proof in \cite{Ni3} used the  theory of analytic currents of Hardt \cite{Ha}.


 
  

 \section{\bf{Superconnections and their Chern character forms}}\label{SCcf}
 
 Superconnections were introduced by Quillen in his celebrated article \cite{Qu0}. In that  same  article, the author makes the case for a currential transgression formula for the Chern character form of a superconnection, although he lives open a precise statement. The theme appears again in another famous work \cite{MQ}, this time  the focus being on the Thom form of a vector bundle. To the best of our knowledge it was Getzler  (\cite{Ge}, Theorem 2.1) who first stated and proved a theorem about the weak convergence of a $1$-parameter family of Thom forms which are built in the spirit of the Mathai-Quillen formalism.  On the other hand, it was probably clear for Harvey and Lawson from their first work on the topic (\cite{HL1})  that their theory should apply to this context as well and in a way, the general results of \cite{HL2} justify that belief. Our purpose in this section is to provide more details about this relation, on one hand as the theory of superconnections is of utmost importance in applications to local index theory, on the other because we wanted to give an answer to Quillen original question. In the next section we will have the occasion to take another look at Getzler's result as well.


We start by reviewing the superconnection formalism but we will be rather sketchy and invite the reader for more  details to the classical \cite{Qu0}, \cite{Qu} and \cite{BGV}.
 
 The most used type\footnote{by no means the only one} of  superconnection on a complex vector bundle  $E\ra B$ is an operator of type $\nabla +A$ where $\nabla$ is connection on $E$ and $A \in \End(E)$. This (differential) operator acts on the space of sections $\Gamma(\Lambda^* T^*M\otimes E)$ by extending the action of $\nabla$  via Leibniz rule and by letting $A$ act only on the second component. 
 
 The word "super" is related to the fact that one assumes in general that $E$  comes with a $\bZ_2$-grading which $\nabla$ respects and $A$ anti-commutes with the involution $\epsilon$ that induces this $\bZ_2$-grading, i.e. $A$ is odd. Then $\bA:=\nabla+A$ is an \emph{odd} operator acting on the sections of the \emph{super} vector space $\Gamma(\Lambda^* T^*M\otimes E)$. The curvature $F(\bA):=\bA^2$ has the same pleasant property as $F(\nabla)$, namely that it is a $0$-order differential operator, i.e. a section of $\Lambda^* T^*M\otimes \End(E)$, an even section to be precise. The same can be said about $e^{F(\bA)}$.  
 
 One has a relevant notion of trace for elements of $\Lambda^* T^*M\otimes \End(E)$ which is called the supertrace:
 \[ \str^+:\Lambda^* T^*M\otimes \End(E)\ra \Lambda^*T^*M,\quad\quad \str^+(\omega\otimes B)=\omega\cdot \tr(\epsilon B).
 \]
 The supertrace is itself an \emph{even} operator  and therefore applied to $e^{F(\bA)}$ will return an even form called the Chern character form of the super connection:
 \[ \ch^+(\bA):=\str^+(e^{F(\bA)})\in\Gamma(\Lambda^{\even}T^*M).
 \]
 One checks that $\ch^+(\bA)$ is a closed form as a consequence of Bianchi's differential identity. 
 
 \begin{remark} If one wants the topological Chern character then one should use $\frac{i}{2\pi}F(\bA)$ instead of $F(\bA)$.
 \end{remark}
 
 Thinking of odd $K$-theory, it is desirable to be able to produce odd forms as well. For that end, one forgets about the grading $\epsilon$ to begin with. Instead one adjoins an \emph{odd} element $\sigma$ which satisfies $\sigma^2=1$. The way to do that is to consider the superalgebra 
 \[ \End(E)[\sigma]:=\End(E)\otimes \bC[\sigma]=\End(E)\oplus\End(E)\sigma.\]
The decomposition is into even and odd elements. 
\begin{remark}
 One can realize $\sigma$ as the involution that flips the factors in $F:=E\oplus E$ and then one has 
 \[\End(E)\otimes \bC[\sigma]\simeq \End_{\bC[\sigma]}(F),\]
  where on the right hand side one considers endomorphisms that   commute with $\sigma$. This becomes an isomorphism of superalgebras ($\End(E)$ sits entirely in even degree) as long as one takes the grading on $F$ to be the one induced by the involution of the decomposition $F=E\oplus E$ (and not by $\sigma$!!!). The elements of $\End_{\bC[\sigma]}(F)$ decompose uniquely as a sum of an even and an odd element:
 \[ \left(\begin{array}{cc} A&0\\
  0& A \end{array}\right)+ \left(\begin{array}{cc} B&0\\
  0& B \end{array}\right)\sigma=\left(\begin{array}{cc} A& B\\ B&A\end{array}\right).\]
  \end{remark}
  The relevant supertrace now is: 
  \[ \str^-: \End(E)[\sigma]\ra \bC,\quad\quad\str(A+B\sigma)=\tr{B}
  \]
   and this extends naturally to an \emph{odd} operator:
   \[ \str^-:\Lambda^*T^*M\hat\otimes\End(E)[\sigma]\ra \Lambda^*T^*M,\quad\quad \str(\omega\otimes A+\eta\otimes B\sigma)=\eta\cdot \tr{B}.
   \]
   We have denoted by $\hat{\otimes}$   the super tensor product of superalgebras. Now, for every endomorphism $A\in\End(E)$, $\bA:=\nabla+A\sigma$ can be realized as a super connection in the sense above. It acts on $\Gamma(\Lambda^*T^*M\otimes F)$ as the operator:
   \[ \left(\begin{array}{cc} \nabla& 0\\
     0 & \nabla\end{array}\right)+\left(\begin{array}{cc} 0 & A \\
      A& 0\end{array}\right).
   \] 
   In any case, one has $F(\bA)=F(\nabla)+ A^2+[\nabla, A]\sigma$, an even element of $\Gamma(\Lambda^*T^*M\hat{\otimes}\End(E)[\sigma])$ and $e^{F(\bA)}$ is still even. However, since the supertrace in this case is an odd operator: 
   \[ \ch^-(\bA):=\str^-(e^{F(\bA)})\in\Gamma(\Lambda^{\odd}T^*M).
   \]
   
  In order to apply Theorem \ref{refhomf} one proceeds as follows.  Assume first that $E$ is endowed with a hermitian metric and that the connection $\nabla$ is compatible with the metric.  The fiber bundle $\pi:P \ra B$ to which we apply  the theory satisfies: 
  \begin{itemize}
  \item[(a)] in the even case, where $E=E^{+}\oplus E^{-}$, the fiber at $b$ is $\Gr_m(E^+_b\oplus E^-_b)$ where $m=\dim{E^+_b}$;
   \item[(b)] in the odd case, where there is no grading, the fiber at $b$ is either the Grassmannian of Hermitian Lagrangians $\Lag(E_b\oplus E_b)$ \emph{or}  $U(E_b)$;  recall that by Arnold Theorem \cite{A}  the two spaces are diffeomorphic (see also \cite{Ci} or \cite{Ni3}).
  \end{itemize}
  \begin{remark}\label{evun} In the even case one can alternatively choose the fiber to be a certain connected component of the space of unitary operators $g\in U(E_b)$ inverted by the grading involution $\epsilon$, meaning that $(\epsilon g)^2=1$. It is not hard to see that this space of unitary operators is diffeomorphic with the full Grassmannian $\cup_{k=0}^{m+l}\Gr_k(E_b)$. This is Quillen's idea in \cite{Qu}.
  \end{remark}
  
  The flow on $P\ra B$ will be the vertical gradient flow of the function $f(L)=\Real\Tr \epsilon P_L$ in both cases where $P_L$ is the orthogonal projection and where $\epsilon$ is the obvious involution of $E\oplus E$ in the odd case. Alternatively, in the odd case, one considers the flow $f(U)=\Real\Tr (U)$ already discussed in Section \ref{occMs}. These flows are really compactifications of the linear flows $(t,A)\ra tA$. The one on $\Gr_m(E^+,E^-)$  is the main object of investigation in \cite{HL2}. We recall its structure.
  
  The critical manifolds are $\bigcup_{k=0}^m F_k$ where:
  \begin{equation}\label{critgr} F_k:=\{L\in \Gr_m(E^+\oplus E^-)~|~\dim(L\cap (E^+\oplus 0))=k,\;\;\dim(L\cap  (0\oplus E^-))=m-k\}
  \end{equation}
 \[\qquad\qquad\simeq  \Gr_k(E^+)\times \Gr_{m-k} (E^-).\]
 
 The stable and unstable manifolds are:
 \[\Sigma_k:=\{L\in \Gr_m(E^+\oplus E^-)~|~\dim{L\cap (E^+\oplus 0)}=k\}
 \]
 \[\Upsilon_k:=\{L\in \Gr_m(E^+\oplus E^-)~|~\dim{L\cap (0\oplus E^-)}=m-k\}
 \]
  
  Part of the datum involved in the construction of the Chern character form is an endomorphism $A\in\End(E)$. We will consider $A$ to be \emph{self-adjoint} at every point and the section $\phi_0$ of $P\ra B$ will be the graph of $A$. To be more precise,
  \begin{itemize}
  \item[(i)] in the odd case, the graph of a self-adjoint operator is clearly a lagrangian subspace of $E\oplus E$, i.e. $J\Gamma_A=(\Gamma_A)^{\perp}$, where $J=\left(\begin{array}{cc} 0 &1 \\ -1 &0\end{array}\right)$; in the unitary group picture the section $\phi_0$ is the Cayley transform of $A$ at $-1$; i.e $\phi_0(b)= \frac{A_b-i}{A_b+i}$.
  \item[(ii)] in the even case, since $A_b$ is odd and self-adjoint it has the block decomposition $A=\left(\begin{array}{cc} 0 & \tilde{A}_b^*\\ \tilde{A}_b& 0\end{array}\right)$; the section will be $\phi_0(b)=\Gamma_{\tilde{A}_b}\in\Gr_k(E^+\oplus E^-)$. 
  \end{itemize}
  
  Finally the form to be flown is constructed as follows.  Let $\pi^*E$ and $\pi^*\nabla$ be the pull-back bundle with connection over $P$. Along the dense open subset of $P$
   \[ \Hom(E^+,E^-)\subset\Gr_k(E^+\oplus E^-) \;\; \mbox{ in the even case, or}\] 
    \[ \Sym(E):=\{A:E\ra E~|~A=A^*\}\subset\Lag(E\oplus E)\;\; \mbox{ in the odd case}\]  there exists a  tautological section $s^{\tau}$ of the vector bundle $\;$ $\Hom(\pi^*E^+,\pi^*E^-)$ and $\Sym(\pi^*E)$, respectively. One can use the superconnection $\pi^*\nabla +s^{\tau}$ (or $\pi^*\nabla +s^{\tau}\sigma$ in the odd case) to build Chern character forms. However that is not good enough as these forms are defined a priori only along the open dense set of $P$ where $s^{\tau}$ makes sense.
    
     First, let us observe  that $s^{\tau}$ makes sense on the whole $P$, but as a section of $\Gr(\pi^*E^+,\pi^*E^-)$ or  $U(\pi^*E)$.  Now Quillen  comes to our rescue again. In \cite{Qu},  he shows that the Chern character forms make sense for such sections and therefore extend to the whole $P$. In fact, using the precise relation between the resolvent and the exponential of a linear operator that the Laplace transform provides, Quillen shows that if one substitutes $A$ by the Cayley transform of a unitary endomorphism $U$ in the formula of the Chern character form above, the resulting form still makes sense for any $U$. This is the odd picture.
     
      In the even case, Remark \ref{evun} allows us to apply essentially the same ideas to the Grassmannian sections as well. To be more precise, an odd, self-adjoint endomorphism $A\in \Sym(E^+\oplus E^-)$ splits as  $A=\left(\begin{array}{cc} 0 & \tilde{A}^*\\ \tilde{A}&0 \end{array}\right)$. The Cayley transform of $A$ is a unitary endomorphism inverted by $\epsilon$.  The correspondence one gets with the Grassmannian is via the map $A\leftrightarrow \Gamma_{\tilde{A}}\in \Gr_m(E^+\oplus E^-)$. The flow line $tA$ ends up when $t\ra \infty$ at the critical point (see (\ref{critgr})):
      \[ L:=\Ker{\tilde{A}}\oplus (\Ker{\tilde{A}^*})^{\perp}=\Ker{\tilde{A}}\oplus \Imag \tilde{A}.
      \]
      Notice that if $\dim{\Ker{\tilde{A}}}=k$ then $\dim{\Ker{A}}=2k-(m-l)$ where $m=\dim{E^+}$ and $l=\dim{E^-}$. We call $m-l=:\ind{E}$ the index of $E$.

   In summary, Quillen's Theorem 1 from \cite{Qu} implies that there exist a \emph{global} form on $P$ whose restriction to the open dense set described in the previous paragraph coincides with the Chern character form associated to $\pi^*\nabla+s^{\tau}$. Notice the crucial fact that $s:B\ra P$ pulls-back this Chern character form to the Chern character form on $B$ built from the original $\bA:=\nabla+A$.  This is  due to the naturality of $\ch$, i.e. if $f:B_1\ra B$ is a smooth map and $(E,\nabla, A)$ is a triple as above on $B$ then:
   \[ f^*\ch(E,\nabla, A)=\ch(f^*E,f^*\nabla, f^*A).
   \]  
   
   We let $\bA_t:=\nabla +tA$ in the even case and $\bA_t:=\nabla+tA\sigma$ in the odd case.
   
  The  question Quillen originally asked in \cite{Qu0} (page 92) was under what  conditions the limit $\displaystyle{\lim_{t\ra  \infty}\str^{\pm}e^{F(\bA_t)}}$ exists as currents.  He proved in the same article that the current has to be supported in the set of points corresponding to operators that have kernel. Theorem \ref{refhomf} gives the complete answer.
  
  \begin{theorem} \label{ansQu} Let $E\ra B$ be a hermitian  bundle endowed with compatible connection $\nabla$. Let $A\in\Sym(E)$ be a  self-adjoint endomorphism. If $E$ is $\bZ_2$-graded, suppose moreover that the $\bZ_2$-grading involution is parallel with respect to $\nabla$ and that $A$ is odd. Let $\ch(A_t)$ be the Chern character form associated to the triple $(E,\nabla, tA)$. Suppose that $A$ is s-normal. This means:
  \begin{itemize}
  \item[(a)] in the odd case, it is transversal to  $S(k):=\{T\in\Sym(E)~|~\dim\Ker{T_b}=k,\; \forall b\}$. 
  \item[(b)] in the even case, it is transversal to $\Sigma_k:=\left\{T=\left( \begin{array}{cc} 0 & \tilde{T}^*\\ \tilde{T}&0 \end{array}\right)~|~\dim{\Ker{\tilde{T}_b}}=k, \; \forall b\right\}$. 
  \end{itemize} Then  
  \begin{itemize}
  \item[(a)] in the ungraded case:
  \[ \displaystyle\lim_{t\ra \infty}\ch(A_t)=\sum_{k\geq 1} \Res_k^-\cdot[A^{-1}(S(k))]\]
  where $\Res_k$ are forms on $A^{-1}(S(k))$ described below;
  
 \item[(b)] in the $\bZ_2$-graded case:
 \[  \displaystyle\lim_{t\ra \infty}\ch(A_t)=\sum_{k\geq 1} \Res_{k}^+\cdot[A^{-1}(\Sigma_k)],
 \]
 where $ \Res_{k,}^+$ are forms on $A^{-1}(\Sigma_k)$ described below.
   \end{itemize}  
  \end{theorem}
  
  \vspace{0.3cm}
  Let $S^-_k:=A^{-1}(S(k))$ in the odd case  and $S^+_k:=A^{-1}(\Sigma_k)$ in the even case.
  
  In order to describe the forms $\Res_k^{\pm}$, let us notice that  over $S^{\pm}_k$ there exists a tautological hermitian vector bundle $\tau$ (which is $\bZ_2$-graded in the even case) with compatible connection $\nabla^{\tau}$. The fiber at a point $b\in S^{\pm}_k$  is $\Ker{A_b}\subset E_b$.  In the even case, $\dim{\Ker {A}_b}=2k-\ind_{E}$ and $\Ker{A}_b$ splits into $\Ker{\tilde{A}}_b\oplus\Ker{\tilde{A}^*_b}$. The connection $\nabla^{\tau}$ is the orthogonal projection of the connection $\nabla$ on $E$ onto $\tau$. 
  
 Let us denote by $\Pi:\Sym(\tau)\ra S^{\pm}_k$ the vector bundle of self-adjoint operators on $\tau$, which are odd in the $\bZ_2$ graded case. The bundle $\Pi^*\Sym(\tau)\ra \Sym(\tau)$ comes with an obvious tautological section $s^{\tau}$. We can build the Chern character form which lives on $\Sym(\tau)$:
  \[ \ch(\Pi^*\nabla^{\tau}+s^{\tau}).
  \]
  \begin{theorem} \label{ansQu2}
  \[\Res_k^{\pm}= \int_{\Sym(\tau)/S^{\pm}_k} \ch(\Pi^*\nabla^{\tau}+s^{\tau}).
  \]
  \end{theorem}
  
  \begin{proof} We prove just the ungraded case, the other one being treated similarly. 
  
    We need the following result from \cite{Qu} (Theorem 2): 
  
  \vspace{0.3cm}
 
 \noindent
{\bf Theorem} (Quillen) Let $E'\hookrightarrow E$ be an isometric embedding of hermitian vector bundles (commuting with $\epsilon$ in the even case) and suppose $U'\in U(E')$ is a unitary endomorphism. Suppose moreover that  $\nabla$ is a connection on $E$ and $\nabla'$ is the connection on $E'$  resulting by orthogonal projection of $\nabla$. Then
 \[\ch(\nabla',U')=\ch(\nabla, U'\oplus \id_{(E')^{\perp}})\footnote{The original statement involves unitary operators $U=U'\oplus -\id_{(E')^{\perp}}$. However, we work with a different Cayley transform throughout.} \]    
  where both sides represent the (extended)  Chern character forms associated to pairs (connection, unitary endomorphism).
  
  \vspace{0.3cm}
  
  The residue to be computed is up to a sign the pull-back (to $A^{-1}(S(k))$) of the form:
  \[\int_{U(k)/ F(k)}\ch(\pi^*E,  \pi^*\nabla, s^{\tau}_u),
  \]
  where $\pi:U(E)\ra B$,  $s^{\tau}_u:U(E)\ra \pi^* U(E)$ is the tautological section, 
   \[U(k)_b:=\{U\in U(E_b)~|~\dim{\Ker{(1-U)}}=n-k\}\quad\]\[ F(k)_b:=\{\id_{L^{\perp}}\oplus -\id_{L}\in U(E_b) \;|\; L\subset  E_b,\; \dim{L}=k \}\quad  \mbox{ and}\]
   \[ \pi_k:U(k)\ra F(k),\qquad U\ra \id_{\Ker{(1-U)}}\oplus -\id_{\Ker{(1-U)}^{\perp}}.
   \]
   
  We identify $F(k)_b$ with the Grassmannian of $k$ subspaces of $E_b$ on which we have a tautological bundle $\tau_b$ and its orthogonal complement $\tau^{\perp}_b$. Now, over $U(k)$, one has 
  \[\pi^*E=\pi_k^*\tau^{\perp}\oplus \pi_k^*\tau.\] 
and the tautological section $s^{\tau}_u\bigr|_{U(k)}$ splits as $s^{\tau}_u=\id\oplus \tilde{s}^{\tau}$ where, by taking its Cayley transform, we think of $\tilde{s}^{\tau}$ as a section of $\pi_k^*{\Sym(\tau)}$. Even better, we can look at $\tilde{s}^{\tau}$ as a section $\tilde{s}^{\tau}:\Sym(\tau)\ra \pi_k^*{\Sym(\tau)}$ since $U(k)\simeq \Sym(\tau)$.

Let $\pi_k^*\nabla^{\tau}$ be the projection of $\pi^*\nabla$ onto $\pi_k^*\tau$. This connection is the same as the pull-back via  $\pi_k$ of the projection of $\nabla$ onto $\tau$. We can now use Quillen's Theorem to conclude that
\[ \ch(\pi^*E, s^{\tau}_u, \pi^*\nabla)=\ch(\pi_k^*\tau, \pi_k^*\nabla^{\tau}+\tilde{s}^{\tau}).
\]
  It is not hard to see now that the pull-back of $\int_{\Sym(\tau)/F(k)}\ch(\pi_k^*\tau, \pi_k^*\nabla^{\tau}+\tilde{s}^{\tau})$ via $\tau_F$ is the same as $ \int_{\Sym(\tau)/S^{\pm}_k} \ch(\Pi^*\nabla^{\tau}+s^{\tau})$.

  \end{proof}

\section{\bf{Thom forms}}\label{sult}

In their celebrated article \cite{MQ}, Mathai and Quillen gave various constructions of equivariant Thom forms by exploiting  the Chern character defect as a ring homomorphism from $K$-theory to singular cohomology. While initially defined for vector bundles with a Spin  (or Spin$^c$) structure (necessary for one to have Thom isomorphism in $K$-theory), these forms make sense on every oriented vector bundle endowed with a Riemannian metric and compatible connection.  The justification of this fact goes through the intricacies of the Weil algebra and equivariant differential forms. A more direct construction of canonical Thom forms on a real vector bundle in the spirit of the superconnection formalism appears in  Getzler's article \cite{Ge}. Before we proceed with his construction, let us just say that Harvey and Lawson gave in \cite{HL1} a recipe for constructing myriads of such forms using as  building blocks   triples made from of a connection, a section and a mode of approximation. Since they already discuss their examples at length in various places we chose Getzler's Theorem 2.1 in the above mentioned article, as an application of the non-compact extension Theorem \ref{nncext}.
  
  Let $\pi:E\ra B$ be an oriented, Riemannian vector bundle endowed with a metric compatible connection $\nabla$.  The curvature $F(\nabla)\in \Gamma(\Lambda^2T^*B\otimes \mathfrak{so}(E))$ can be seen as a two form with values in $\Lambda^2E$ by using the canonical bundle isomorphism:
  \[ \mathfrak{so}(E_b)\simeq \Lambda^2E_b,\qquad A\ra \sum_{1\leq i<j\leq n}\langle e_i,Ae_j\rangle e_i\wedge e_j,
  \]
  induced by an orthonormal basis $\{e_1,\ldots,e_n\}$ of $E_b$.
 
 One pulls back $\nabla$ to $\pi^*E$. The bundle $\pi^*E\ra E$ has a natural tautological section $x$. Define the following element of $\Gamma(E;\Lambda^*T^*E\otimes\Lambda^*\pi^*E)$ for every $t\geq 0$:
 \[\omega_t:=\frac{t^2}{2}|x|^2+t\pi^*\nabla(x)+\pi^*F(\nabla).
 \]
 This is the analog of the curvature of a superconnection. In order to get "honest" forms on $E$ one needs a trace, which in this context is represented by  the Berezin integral. The metric and the orientation of $E$ give a bundle isomorphism:
 \begin{equation}\label{berint} \Lambda^{\max}E\simeq \underline{\bR}.
 \end{equation}
 Berezin integral  is the bundle morphism $\mathscr{B}:\Lambda^*E\ra \underline{\bR}$ equal to the  isomorphism (\ref{berint}) on $\Lambda^{\max}E$ and with $0$ everywhere else. It lifts to a bundle morphism $\Lambda^*\pi^*E\ra \underline{\bR}$. Now, assume $E$ is of even rank and define the following  form on $E$:
 \[ \mu_t:=(-1)^{n(n-1)/2}(2\pi)^{-n/2} \mathscr{B}(e^{-\omega_t}).
 \]
\begin{lemma} The form $\mu_t\in\Gamma(E;\Lambda^{n}T^*E)$ is closed.
\end{lemma}
\begin{proof} See Proposition 1.3 (1) in \cite{Ge}. 
\end{proof}
\begin{remark} Notice that for $t=0$ one gets the Pfaffian of the connection $\pi^*\nabla$.
\end{remark}

The vertical flow that we consider on $E\ra B$ is the radial flow $v\ra tv$ induced by the gradient  of fiberwise Morse function $f(v)=\frac{1}{2}|v|^2$. There is one stable manifold in the story and that is the zero section and one unstable manifold and that is the entire $E$. There is only one residue  to compute and this is 
\[ \int_{E/B} \mu_t
\]
which one can see that is in fact the computation of a Gaussian in every fiber as follows from Proposition 1.3(2) in \cite{Ge}. Hence $\mu_t$ is a Thom form for every $t$. 

Let $s:B\ra E$ be a section transversal to the zero section. Theorem \ref{nncext} gives  the  Chern-Gauss-Bonnet Theorem again, in the form of Theorem 2.1 in \cite{Ge}.
\begin{theorem}[Getzler] The forms $s^*\mu_t$ converge as $t\ra \infty$ to a current of the form:
\[ \lim_{t\ra \infty} s^*\mu_t=[s^{-1}(0)].
\]
\end{theorem}
The original proof was based on an ad-hoc argument relying on the Lebesgue Dominated Convergence Theorem.

\begin{appendix}

\section{Resolutions of Bott-Samelson type with corners} \label{appA}
 In  \cite{La},  Latschev proves that the closure of any stable or unstable manifold of a Morse-Bott -Smale flow on a compact manifold has finite volume. The main technical tools can be traced back to the proof of Theorem 14.3 in \cite{HL3}. We choose to present  here most of the details, instead of just citing the article for several reasons. First and foremost, we need the details for the proof of Theorem \ref{gan} and of Proposition \ref{essprop2}.  Second, while the main ideas are the same, the presentation below is slightly different, simplifying certain proofs and allowing for a more direct formalization of the arguments. An example of this  is the discussion of the refinement of  the smooth structure  one is forced  to consider at a certain point on the product of two manifolds with corners (compare with  Lemma 3.4 in \cite{La}).

 The main result of Latschev affirms the existence of  Bott-Samelson resolutions for the stable and unstable manifolds. The domains of these resolutions are \emph{manifolds with corners}. More precisely one has the following:

\begin{theorem}[Latschev] \label{Lath} Let $P$ be a compact manifold with a Morse-Bott-Smale function $f$ on it. Consider the gradient flow of $f$. The closure of any unstable  manifold is the image of a smooth map whose domain is a manifold with corners and whose target is $P$. This map is a  diffeomorphism from the interior of the manifold with corners to an open dense set of the stable manifold. An analogous result holds true for any stable manifold.
\end{theorem}
 The finite volume of the stable/unstable manifolds of the gradient of $f$ is an obvious corollary of this theorem. 
 
 The main technical tool used by Latschev is the blow-up of the stable spheres of critical manifolds in order to let the flow lines flow "without stops". The strategy is as follows. Organize\footnote{Technically one has to work with \emph{gradient like} vector fields. The essential arguments are the same though.} first the critical manifolds to lie at given energy\footnote{We follow Milnor in calling $f$ the energy function even if the appropriate notion is the  action functional. We do it for the sake of suggestive expressions  like "higher/lower energy".} levels. In order to show that a fixed unstable manifold $U(F)$ (where  $F$ is supposed for simplicity to lie at level $0$) is covered by a manifold with corners one starts by noticing that  one has a natural  surjective smooth map from $[0,a]\times U_{a}(F)$ to $U_{\leq a}(F)$, where $U_{\leq a}(F)$ is  $f^{-1}([0,a])\cap U(F)$ and $U_{a}(F):=f^{-1}(a)\cap U(F)$.   This is true at least for $a$ smaller than the next critical level. The rest goes roughly as follows.
 
 Let $W_1:=U_{a}(F)$ and consider the canonical inclusion $\iota:W_1\hookrightarrow f^{-1}(a)$. Let $\widetilde{W}_1$ be the blow-up of  $S(F')\cap W_1$ inside $W_1$ and  let $\tilde{V}_a$ be the blow-up of $S(F')\cap f^{-1}(a)$ inside $f^{-1}(a)$.  The inclusion $\iota$ lifts to a map $\tilde{\iota}:\widetilde{W}_1\ra \tilde{V}_a$ because of the transversality of $U(F)$ with $S(F')$. It turns out that  $\tilde{V}_a$ is just the $a$-slice of a more general blow-up $\tilde{V}$ of the union  $S(F')\cup U(F')$ inside the neighborhood of type $f^{-1}([a, a+b])$, $b>0$ of $F'$. Here, $f(F')=a+c$ with $c\in (0,b)$ and $a+b$ is a regular value coming right after $a+c$ .  Moreover, one has an extension of  $\tilde{\iota}$ to a map
  \[ \hat{\iota}: [a,a+b]\times \widetilde{W}_1\ra \tilde{V}\]
   which is smooth away from $\{a+c\}\times \widetilde{W}_1$. We have to refine the manifold with corners structure of the product $ [a,a+b]\times \widetilde{W}_1$ in order to turn $\hat{\iota}$ into a smooth map. But the crucial point here is that due to the  Smale property of the flow, $\hat{\iota}_{a+b}:\widetilde{W}_1\ra \tilde{V}_{a+b}$ is completely transversal to all $S(F'')$ for all critical manifolds $F''$ that lie further down the flow. Better said, it is transversal to the preimage of $S(F'')$ via the blow-up projection $\tilde{V}_{a+b}\ra f^{-1}(a+b)$. And the process can go on by taking $W_2:=\widetilde{W}_1$ and $\widetilde{W}_2$ to be the blow-up of $\hat{\iota}_{a+b}^{-1}(S(F''))$  inside $\widetilde{W}_1$.
 
 In this appendix we explain the inductive step. We describe first  the blow-up of a union $S(F)\cup U(F)$ in the coordinates of the Morse-Bott Lemma  \ref{MBL}. Then (see Definition \ref{comptrans}) we consider a generic smooth map $\sigma$ from a manifold with corners $W$ to a certain level set $f^{-1}(-\delta)$ that lies before  $F\subset f^{-1}(0)$. The map $\sigma$ is assumed to be transversal (along every corner-strata) to all $f^{-1}(-\delta)\cap S(F')$, including $F'=F$.  We describe how to build a map   $\tilde{\sigma}: [-\delta,\delta]\times\widetilde{W} \ra f^{-1}[-\delta,\delta]$, where $\widetilde{W}$ is the blow-up of $\sigma^{-1}(S(F))$ inside $W$. The map takes $\{t\}\times \widetilde{W}$ to $f^{-1}(t)$. The restriction $\tilde{\sigma}_{-\delta}$ coincides with $\sigma$  away from the blow-up locus  and the restriction $\tilde{\sigma}_{\delta}$ to $\{\delta\}\times \widetilde{W}$ is completely transversal to all  $S(F')\cap f^{-1}(\delta)$.
 
 \vspace{0.3cm}
 
 We present the details now. The difficulties are local around the critical manifolds.  Let 
\[ f:\bR^k\times \bR^{m}\times \bR^p\ra \bR,\quad\quad f(\underline{x},\underline{y},\underline{z})=\frac{1}{2}\left(\sum_{i=1}^kx_i^2-\sum_{j=1}^{m}y_j^2\right),
\]
be a Morse-Bott function with a unique critical manifold represented by the subspace $\underline{x}=\underline{y}=0$. Its gradient flow, $\gamma: \bR\times \bR^{k+m+p}\ra \bR^{k+m+p}$  is easy to describe:
\begin{equation}\label{gamfl} \gamma(t,\underline x,\underline y,\underline{z})=(e^t\underline x,e^{-t}\underline{y},\underline{z}).
\end{equation}

We warm-up by considering the two dimensional version of the blow-up we aim to analyze. For every $\delta>0$  let
\[H_{-\delta,\delta}:=\{(r,s)\in \bR^2~|~r\geq 0,\; s\geq 0,\; rs=1,\; -\delta \leq \frac{1}{2}\left( r^2-s^2\right)\leq \delta\}\]
 be a piece of the standard hyperbola in $\bR^2$ lying in the first quadrant between the levels of energy $-\delta$ and $\delta$. We use it to "blow-up" the broken segment/trajectory:
 \begin{equation}\label{eq0} \{r=0,\;  \sqrt{\delta}\geq s\geq 0\}\sqcup\{ s=0,\; 0\leq r\leq \sqrt{\delta}\}
 \end{equation}
as follows. Let $\psi:H_{-\delta,\delta}\times [0,\epsilon]\ra \bR^2_{x\geq 0,y\geq 0}$ be the map:
\begin{equation}\label{eq1}  \psi(t,q)= \left(\sqrt{\sqrt{t^2+q^2}+t} ,\sqrt{\sqrt{t^2+q^2}-t}\right), 
\end{equation}
where the "energy" coordinate $t=\frac{1}{2}(r^2-s^2)\in[-\delta,\delta]$ uniquely identifies a point on the hyperbola $H_{-\delta,\delta}$. The map $\psi$ takes a point $(t,q)$ to a point $(x,y)$ that lies on a certain hyperbola indexed by $q$ but has the same level of "energy" as $t$. This means that
\begin{itemize}
\item[(i)] $xy=q\in[0,\epsilon]$;
\item[(ii)] $\frac{1}{2}\left(x^2-y^2\right)=t=\frac{1}{2}\left(r^2-s^2\right)$.
\end{itemize}
The image of $\psi$ represents a continuous deformation of (a piece of) the hyperbola $xy=\epsilon$ to the broken trajectory (\ref{eq0}).

 Notice that $\psi$ is not differentiable at $(t,q)=(0,0)$. However  we can change the smooth structure of $H_{-\delta,\delta}\times [0,\epsilon]$ at exactly this point by introducing a corner hence turning this space into something resembling a curved pentagon, on which the map $\psi$ becomes smooth. The trivial way to do that is by noticing that $\psi$ is a homemorphism onto its image. We put the manifold with corners structure on $H_{-\delta,\delta}\times [0,\epsilon]$ that makes  $\psi$ a smooth chart an denote this smooth structure by
\[ H_{-\delta,\delta}\times_{\psi}[0,\epsilon].
\]

We use  now  the model of the blow-up of the broken trajectory \ref{eq0}  to build an analogous family blow-up of all the broken trajectories of $\nabla f$  that pass through the origin.  Let 
\begin{equation}\label{eqap1} V=\{(\underline {x},\underline{y},\ul{z})\in\bR^k\times \bR^{m}\times \bR^p~|~-\delta\leq f(x,y)\leq \delta,\;|x|\cdot |y|\leq \epsilon\}\end{equation} be a neighborhood of the critical manifold $\{\ul{x}=\ul{y}=0\}\subset\bR^{k+m+p}$. We consider a fiberwise  version of the blow-up process above\footnote{This is not quite the "standard" blow-up as the blow-up locus is not a manifold.  It is a union of two manifolds that intersect transversally, which is  the next best thing so to speak.} for  $V_0:=\{|x|\cdot |y|=0\}\cap V$ inside $V$. The word "fiberwise" refers to the projection of $V$ and/or $V_0$ onto the $\ul{z}$-coordinate. This blow-up is given by the following map:
\begin{equation}\label{apcommdiag1} \Psi:H_{-\delta,\delta}\times [0,\epsilon]\times S^{k-1}\times S^{m-1}\times \bR^p\ra \bR^k\times \bR^{m}\times \bR^p,
\end{equation}
\[\Psi(t,q, \underline {x},\ul{y},\ul{z})= \left(\ul{x}\sqrt{\sqrt{t^2+q^2}+t}  ,\;\ul{y}\sqrt{\sqrt{t^2+q^2}-t},\; \ul{z}\right)
\]
One can check easily that $\Imag\Psi=V$ and $\Imag\Psi\bigr|_{q=0}=V_0$ and 
$\Psi\bigr|_{q\neq 0}$ is a diffeomorphism onto $V\setminus V_0$.
\begin{remark} \label{flowlinepsi} There is one other very important point of view on  $\Psi$. Suppose $(x,y,z)\in \bR^k\times \bR^{m}\times \bR^p$ such that $|x||y|=:q\neq 0$. Then the flow line determined by $(x,y,z)$ does not end at a critical point in this chart of $F$. So what are the coordinates of the point of intersection of this flowline with the level set $f^{-1}(t)$? A simple computation using $\gamma$ (see (\ref{gamfl})) gives the answer:
\[ \left(\frac{\ul{x}}{|x|}\sqrt{\sqrt{t^2+q^2}+t},\frac{\ul{y}}{|y|}\sqrt{\sqrt{t^2+q^2}-t},\;\ul{z} \right)=\Psi\left(t, |x| |y|,\frac{\ul{x}}{|x|},\frac{\ul{y}}{|y|},\ul{z}\right).
\] 
\end{remark}

\vspace{0.3cm}

Now, let 
 \[\tilde{V}:=H_{-\delta,\delta}\times [0,\epsilon]\times S^{k-1}\times S^{m-1}\times \bR^p=\Dom{( \Psi)}\] 
 and  $\tilde{V}_t\subset \tilde{V}$ be the level set that freezes the first coordinate at $t$. Notice that $\Psi(\tilde{V}_t)\subset f^{-1}(t)$. 
 
 The map $\Psi$ is not differentiable  exactly at   the points $(0,0,x,y,z)\in\tilde{V}$. In order to turn $\Psi$ into a differentiable map one takes $\tilde{V}$ to be the product manifold structure of $H_{-\delta,\delta}\times_{\psi}[0,\epsilon]$ above with the standard atlas of $S^{k-1}\times S^{m-1}\times \bR^p$. We denote this manifold with corners by $\tilde{V}_{\psi}$.
 
\vspace{0.3cm}

The set
 \[ B_{-\delta}:=\{(0,y,z)~|~|y|^2=2\delta\}\subset\{0\}\times \bR^{m}\times \bR^p\] is the stable sphere bundle (of the unique critical manifold of $f$) lying at level $f=-\delta$.  Let $V_{-\delta}:=V\cap f^{-1}(-\delta)$ be a neighborhood of $B_{-\delta}$, as subsets of $f^{-1}(-\delta)$. Notice that $B_{-\delta}=V_0\cap f^{-1}(-\delta)\subset V_{-\delta}$.

 The map $\Psi\bigr|_{t=-\delta}:\tilde{V}_{-\delta}\ra V_{-\delta} $ can be seen as the real blow-up of the submanifold $B_{-\delta}$ of $V_{-\delta}\subset f^{-1}(-\delta)$. Since both $B_{-\delta}$ and $V_{-\delta}$ are fiber bundles over $\bR^p$, the same map can alternatively  be regarded as the fiberwise  blow-up with respect to $z$. Notice that the piece of the boundary of $\tilde{V}_{-\delta}$ that gets mapped to $B_{-\delta}$ is a (trivial) fiber bundle over $\bR^p$  with fiber $S^{k-1}\times S^{m-1}$. This fiber  is diffeomorphic to the spherical normal bundle of $B_{-\delta}\cap \{z=\ct\}$ as a submanifold of $f^{-1}(-\delta)\cap \{z=\ct\}$, as  usually happens in a real blow-up.
 
 \begin{remark}\label{totblup} While $\Psi$ is the blow-up of $V_0$ inside $V$, we can use $\Psi$ to build a blow-up of $V_0$ inside $f^{-1}([-\delta,\delta])$ by defining the following manifold with corners:
\[\widetilde{f^{-1}[-\delta,\delta]}:=f^{-1}([-\delta,\delta])\sqcup \tilde{V}/\sim,
\]
where $\sim$ identifies a point $b\in f^{-1}([-\delta,\delta])$ with $a$, $a\in\tilde{V}$ provided $a^q\neq 0$ and $\Psi(a)=b$. One has a natural surjective smooth map
 \begin{equation}\label{eqap5} \widetilde{\Psi}: \widetilde{f^{-1}[-\delta,\delta]}\ra f^{-1}([-\delta,\delta]).\end{equation} It also makes sense to speak about the $t$-level set of $\widetilde{f^{-1}[-\delta,\delta]}$ since in the gluing process one identifies points in the $t$-level set of $f$ with points with the first coordinate equal to $t$ in $\tilde{V}$ for every $t\in [-\delta,\delta]$.
\end{remark}


\vspace{0.3cm} 

Suppose now  we are given a smooth map $\sigma: W\ra V_{-\delta} \subset f^{-1}(-\delta) $ defined on a manifold with corners $W$, endowed with a Riemannian metric  such that $\sigma$ is \emph{completely transverse} to $B_{-\delta}$. This means the following.  Let $W:=\bigcup W_i$ be the decomposition of $W$ into smooth strata, each $W_i$ representing the codimension $i$ boundary of $W$.
\begin{definition}\label{comptrans} The map $\sigma$ is said to be completely transverse to $B_{-\delta}$  if the restriction $\sigma\bigr|_{W_i}$   is transverse to $B_{-\delta}$.\end{definition}
One shows by using the Implicit Function Theorem for manifolds with corners that $~$  $\sigma^{-1}(B_{-\delta})$ is a submanifold with corners of codimension $l=\codim B_{-\delta}$. The corner stratum of $\sigma^{-1}(B_{-\delta})$ of codimension $i$ lies inside $W_i$. One can therefore speak about the normal bundle of $\sigma^{-1}(B_{-\delta})$ which is isomorphic to the pull-back of the normal bundle $N B_{-\delta}$ of $B_{-\delta}$ as a submanifold of $V_{-\delta}$.

  Let $\Bl:\widetilde{W}\ra W$ be the result of performing a standard  blow-up of $\sigma^{-1}(B_{-\delta})$ inside $W$. By definition,
   \begin{equation}\label{bldiag}\widetilde{W}=W\setminus \sigma^{-1}(B_{-\delta})\sqcup S(N\sigma^{-1}(B_{-\delta}))\times [0,1)/\sim^{\beta},\end{equation} where $S(N\sigma^{-1}(B_{-\delta}))$ is the spherical normal bundle and $\beta:S(N\sigma^{-1}(B_{-\delta}))\times[0,1)\ra W$ is a smooth map whose image coincides with a tubular neighborhood of $\sigma^{-1}(B_{-\delta})$ and such that $\beta_0$ coincides with the bundle projection $S(N\sigma^{-1}(B_{-\delta}))\ra \sigma^{-1}(B_{-\delta})$. An example of such a map  $\beta$ is:
   \[  \beta(q,v,t)=\exp_q(vt),\quad\forall q\in \sigma^{-1}(B_{-\delta}),\; |v|=1,\;  t\in[0,1)
   \] Finally $p\sim^{\beta}(q,v,t)\Leftrightarrow \beta(q,v,t)=p$.  
      
  Since $N\sigma^{-1}(B_{-\delta})=\sigma^*N B_{-\delta}$ there exists a lift of $\sigma$ that makes the diagram commutative.
\begin{equation}\label{apcommdiag}  \xymatrix{ \widetilde{W}\ar^{\tilde{\sigma}_{-\delta}}[r]  \ar[d]_{\Bl}& \tilde{V}_{-\delta}\ar[d]^{\Psi_{-\delta}} \\
W\ar[r]^{\sigma} & V_{-\delta} }
\end{equation}
The map $\tilde{\sigma}_{-\delta}$ takes the exceptional divisor $\Bl^{-1}(\sigma^{-1}(B_{-\delta}))$ (which is also the newest "boundary") into the exceptional divisor $\tilde{V}_{-\delta}\cap \{q=0\}$ and equals $(\Psi_{-\delta,q\neq 0})^{-1}\circ \sigma\circ \Bl$ away from those points.

One can construct  an extension of the map $\tilde{\sigma}_{-\delta}$ simply:
\begin{equation}\label{apdiag2} \tilde\sigma: [-\delta,\delta]\times \widetilde{W}\ra \tilde{V},\quad\quad \tilde{\sigma}(t,\tilde{w}):=(t,\tilde{\sigma}_{-\delta}(\tilde{w})).\footnote{
In this definition, $\tilde{\sigma}_{-\delta}$ stands for the last $4$ ("non-trivial")  coordinates of $\tilde{\sigma}_{-\delta}$ in $\tilde{V}_{-\delta}$.}\end{equation}

The map $\tilde{\sigma}$  is not differentiable when one considers the atlas $\tilde{V}_{\psi}$ on the right hand side. The points where differentiability fails are of type $(0, \tilde{w})$ with $\tilde{w}\in \Bl^{-1}(\sigma^{-1}(B_{-\delta}))$. 

However, we can refine the product smooth structure on $[-\delta,\delta]\times \widetilde{W}$ by introducing new corners at the non-differentiability points in order to make $\tilde{\sigma}$ smooth. First, there exists an open neighborhood of $\Bl^{-1}(\sigma^{-1}(B_{-\delta}))$  in $\widetilde{W}$ of product type $[0,1)\times \Bl^{-1}(\sigma^{-1}(B_{-\delta}))$. This is immediate from the way the blow-up is built. We  modify the product smooth structure on $[-\delta,\delta]\times[0,1)\times \Bl^{-1}(\sigma^{-1}(B_{-\delta}))$ by introducing just one new corner in $[-\delta,\delta]\times[0,1)$ at the point $(0,0)$ analogous to what was done above with $H_{-\delta,\delta}\times_{\psi}[0,\epsilon)$ and then we take the product with $\Bl^{-1}(\sigma^{-1}(B_{-\delta}))$. This is the manifold with corners structure we consider on $[-\delta,\delta]\times \widetilde{W}$ for which the map $\tilde{\sigma}$ is smooth.
\begin{remark} We took the image of $\sigma$ to be a subset of $V_{-\delta}$ to better emphasize what happens around the critical manifold. However, it is important for the induction step to consider, more generally, maps $\sigma:W\ra f^{-1}(-\delta)$. Since $V_{-\delta}\subset f^{-1}(-\delta)$ is an open subset, there exists an inclusion map $\tilde{V}_{-\delta}\ra \widetilde f^{-1}(-\delta)$ where the later space is the blow-up of $B_{-\delta}$ inside $f^{-1}(-\delta)$. For the same reasons as above there exists a lift $\tilde{\sigma}_{-\delta}:\widetilde{W}\ra\widetilde {f^{-1}(-\delta)}$. Moreover, there exists a map $\tilde{\sigma}:[-\delta,\delta]\times \tilde{W}\ra \widetilde {f^{-1}[-\delta,\delta]}$ (see Remark \ref{totblup}). Indeed, for every point $(t,p)$ with $p\notin \Bl^{-1}(\sigma^{-1}(B_{-\delta}))$, $\tilde{\sigma}(t,p)$ is the point at the $t$-level of $\widetilde {f^{-1}[-\delta,\delta]}$ obtained by first considering $p_t$, the point of intersection of the flow line $\sigma(p)$ determines with the $t$-level set of $f$ and then taking (see \ref{eqap5}): $\widetilde{\Psi}^{-1}(p_t)\in \widetilde {f^{-1}[-\delta,\delta]}$. This extends smoothly to points $p\in \Bl^{-1}(\sigma^{-1}(B_{-\delta}))$ because of Remark \ref{flowlinepsi}.
\end{remark}

\vspace{0.3cm}
The map to look at is $\Psi \circ \tilde{\sigma}:[-\delta,\delta]\times \widetilde{W}\ra V$.
The crucial property of $\Psi\circ\tilde{\sigma}$ is that it preserves the transversality.  This means the following. Recall that  $\sigma:W\ra  V_{-\delta}$  is completely transverse to $B_{-\delta}=S(F)\cap f^{-1}(-\delta)$.
\begin{lemma} \label{transvlemma} Suppose that  $\sigma$ is also completely transverse to $S(F')\cap f^{-1}({-\delta})$ where $F'$ is another critical manifold. Then $\Psi_{\delta}\circ\tilde{\sigma}_{\delta}$ is completely transverse to $S(F')\cap f^{-1}({\delta})$ where  $\Psi_{\delta}:\tilde{V}_{\delta}\ra V_{\delta}$ is the restriction of $\Psi$.\end{lemma}
\begin{proof} Transversality should be clear for points $p\in S(F')\cap f^{-1}({\delta})$ which do not lie in the closed set 
\[A_{\delta}:=\{(x,0,z)~|~|x|^2=2\delta\} =U(F)\cap f^{-1}(\delta).\] Indeed,  the flow $\gamma$ induces a diffeomorphism between $f^{-1}(-\delta)\setminus B_{-\delta}$ and $f^{-1}(\delta)\setminus A_{\delta}$ and the flow preserves the transversality of the manifolds. By Remark \ref{flowlinepsi} this diffeomorphism has the expression:
\[ (x,y,z)\ra \Psi_{\delta}\left(|x||y|,\frac{x}{|x|},\frac{y}{|y|},z\right).
\]

 Combine this with  the fact that away from the blow-up locus one has:
  \[ \tilde{\sigma}_{\delta}=\left(\delta, |\sigma^x|\cdot|\sigma^y|, \frac{\sigma^x}{|\sigma^x|},\frac{\sigma^y}{|\sigma^y|},\sigma^z\right)\] and one gets the claim.

So far we have not used the Smale transversality property of the flow. We use it now. We fix a corner-stratum $W_i\subset W$ and take a look at $\tilde{\sigma}_{-\delta}\bigr|_{\widetilde {W}_i}$, where $\widetilde{W}_i$ is the result of blowing-up $\sigma^{-1} (B_{-\delta})$ as a submanifold of $W_i$. The  transversality of $\sigma\bigr|_{W_i}$ with $B_{-\delta}\subset V_{-\delta}$ implies that  $d\sigma: N\sigma^{-1}(B_{-\delta})\ra \sigma^*N{B_{-\delta}}$ is a bundle isomorphism along the manifold $\sigma^{-1}(B_{-\delta})\cap W_i\subset W_i$.

The interest in $d\sigma$ comes from the fact that the map $\tilde{\sigma}_{-\delta}$ restricted to the exceptional divisor
 \[  \widetilde{W}_i\cap \Bl^{-1}(\sigma^{-1}(B_{-\delta}))= S(N\sigma^{-1}(B_{-\delta})), \] can be identified with $d\sigma$ restricted to $S(N\sigma^{-1}(B_{-\delta}))$, the spherical normal bundle. Notice that the exceptional divisor in $\tilde{V}_{-\delta}$ is:
  \[(\Psi_{-\delta})^{-1}(B_{-\delta})= S(N{B_{-\delta}})=\{(-\delta, 0, v,w,z)~|~ (v,w,z)\in S^{k-1}\times S^{m-1}\times \bR^p\} \;\mbox{and}\]

Let us now take $p=(x_0,0,z_0)\in S(F')\cap A_{\delta}$ and suppose $p\in \Imag (\Psi_{\delta}\circ \tilde\sigma_{\delta}\bigr|_{\widetilde{W}_i})$. The sphere $A_{\delta}\cap \{z=z_0\}$ which contains $p$ is transverse to $S(F')\cap f^{-1}(\delta)$ by the Smale property. We show that this sphere is in the image of $\Psi_{\delta}\circ \tilde\sigma_{\delta}\bigr|_{\widetilde{W}_i}$  and we are done.

On one hand,
\[\Psi_{\delta}^{-1}(A\cap\{z=z_0\})=\{\delta,0,v,w,z_0)~|~(v,w)\in S^{k-1}\times S^{m-1}\}.
\]
Moreover:
\[  \tilde{\sigma}_{\delta}^{-1}\{(\delta,0,v,w,z_0)\}=(\tilde{\sigma}_{-\delta})^{-1}\{(-\delta,0,v,w, z_0)\}\qquad \forall(v,w)\in S^{k-1}\times S^{m-1}.
\]
We conclude that $p\in  \Imag (\Psi_{\delta}\circ \tilde\sigma_{\delta}\bigr|_{\widetilde{W}_i})$ implies that there exists $w_0\in S^{m-1}$ such that:
\[ (-\delta,0, x_0/\sqrt{2\delta},w_0,z_0)\in\Imag \tilde{\sigma}_{-\delta}\bigr|_{\widetilde{W}_i}.
\]
Clearly such a point $(-\delta,0, x_0/\sqrt{2\delta},w_0,z_0)\in \tilde{V}_{-\delta}$ is a point in the exceptional divisor of the blow-up of $B_{-\delta}$, i.e. in $S(NB_{-\delta})$. Since $\tilde{\sigma}_{-\delta}$ is the spherical normal bundle isomorphism $d\sigma$ we conclude that all points $\{(-\delta,0, v,w_0,z_0)~|~v\in S^{k-1}\}$ are in the image of $\tilde{\sigma}_{-\delta}$ and therefore all points $\{(\delta,0, v,w_0,z_0)~|~v\in S^{k-1}\}$ are in the image of $\tilde{\sigma}_{\delta}$. Hence $A_{\delta}\cap \{z=z_0\}$ is in the image of $\Imag (\Psi_{\delta}\circ \tilde\sigma_{\delta}\bigr|_{\widetilde{W}_i})$.

Finally, notice that the arguments are not affected by the change of smooth structure on $[-\delta,\delta]\times \widetilde{W}$ as they do not involve the points where this change happens.
\end{proof}

The proof of Theorem \ref{Lath} for $U(F)$ with $F$ lying at level $0$ is based on the construction of the Bott-Samelson resolution with corners which starts, as we said,  with $[0,\delta]\times U_{\delta}(F)$ and its natural projection to $U_{\leq\delta}(F)$. In this context we take $W_1=U_{\delta}(F)$, the spherical unstable bundle of $F$ and use the flow to  extend the previous map to one $[0, c-\delta']\times W_1\ra U_{\leq c-\delta'}(F)$ where $c$ is the next critical level at which lies $F'$. Now look at the map $\sigma: W_1\ra f^{-1}{(c-\delta'})$ which is completely transversal to all stable bundles, in particular $S(F')$. Apply the blow-up theory above to "extend" the map $[0, c-\delta']\times W_1$ to a map $[0,c+\delta']\times \widetilde {W_1}\ra f^{-1}([0,c+\delta'])$. Let  $W_2:=\widetilde{W}_1$ and continue this process until there are no more blow-ups to do. The process terminates because the manifold is compact. The Bott-Samelson resolution of $U(F)$ is $[0,\max{f}]\times W_{k}$ with the manifold with corners structure that comes by introducing new corners at the critical levels as in the discussion above.

 \begin{figure}
 \includegraphics[width=300pt]{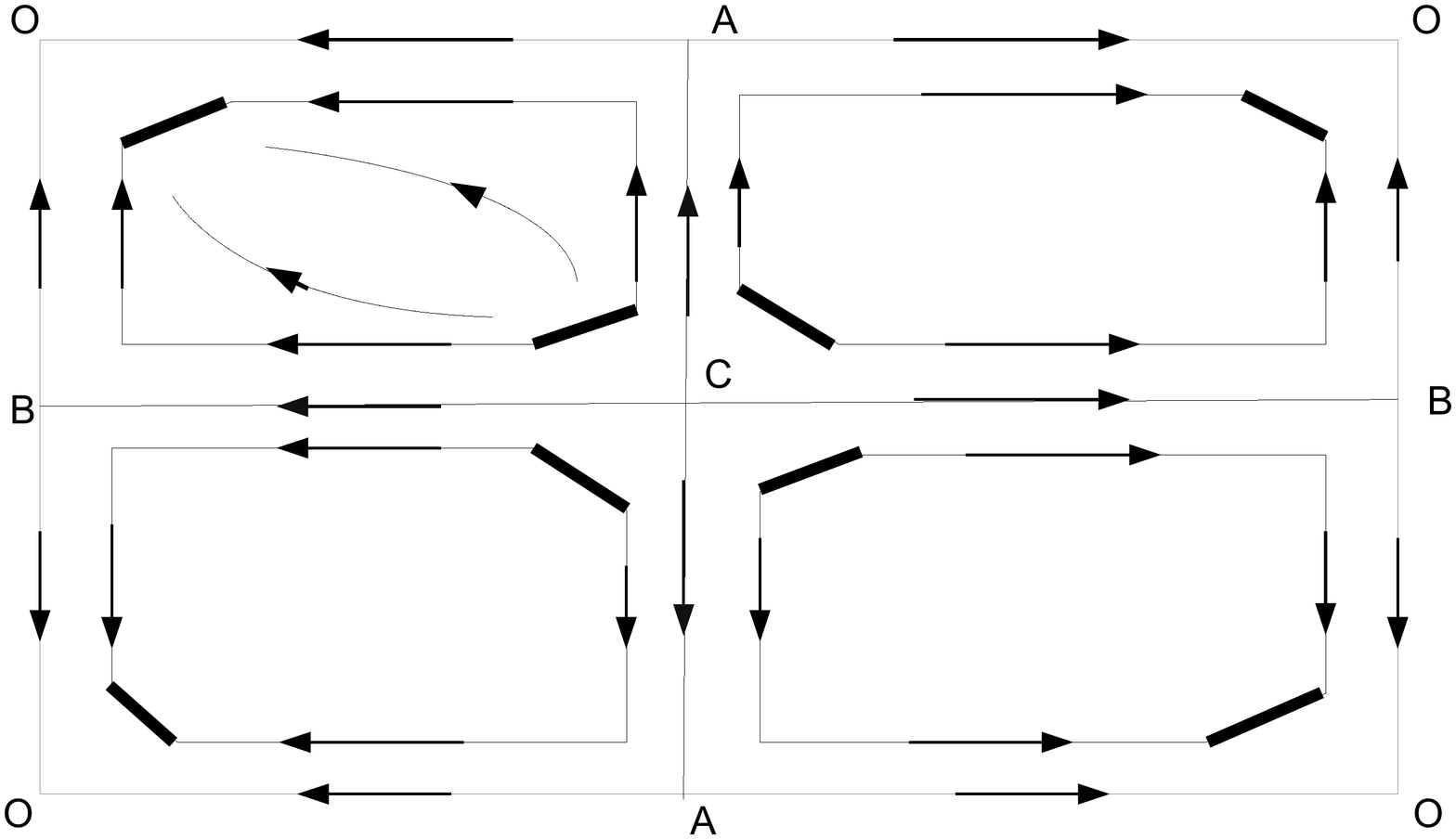}
 \caption{The  flow resolution of the standard flow on the torus with $4$ critical points is made  of a disjoint union of $4$ hexagons. The top and bottom levels of the hexagons are critical while the corners at "midlevel" are a result of the refinement of the smooth structure.}
 \end{figure}

The next result which is a reward of the discussion above gives something of a flow-resolution of a manifold endowed with a Morse-Bott-Smale function. By this we mean that there exists a manifold with corners covering the original manifold and a "linear" flow which covers the original flow.  Standard statements about the compactification of the moduli space of trajectories could, in principle, be derived out of it (compare with \cite{BFK}).  We just sketch the proof since we do not need it. 
\begin{theorem}\label{notn} Let $P$ be a compact manifold  of dimension $n$ and $f$ a Morse-Bott-Smale function on $P$ with each critical manifold lying  entirely within a fixed critical level set of $f$.  Then there exists a manifold with corners $W=\bigcup_{i=0}^{n-1} W_i$ of dimension $n-1$, a  manifold with corners structure on  $[\min{f},\max{f}] \widetilde\times W$ that refines the product structure, and a smooth surjective map $~$ $\pi:[\min{f},\max{f}] \widetilde\times W\ra P$ such that the following hold:
\begin{itemize}
\item[(a)]  each point $p$ in the corner stratum $W_i$ determines a unique broken trajectory $\alpha$ in $P$, trajectory broken exactly $i$-times between $\min{f}$ and $\max{f}$; the curve $t\ra \pi(t, p)$ gives a bijection  $[\min{f},\max{f}]\simeq \alpha$. This curve is smooth if one considers $[\min{f},\max{f}]\times\{p\}$ as a submanifold (with corners) of $[\min{f},\max{f}] \widetilde\times W$. This has the effect of introducing new corners in $[\min{f},\max{f}]$ at exactly the critical levels where  $\alpha$ breaks.
\item[(b)] the restriction of the map $\pi$ to $(\min{f},\max{f}) \times W_0$ is a diffeomorphism onto the open, dense set of $P$ which consists of all (simple) trajectories  connecting   $\min f$ to $\max{f}$. 
\item[(c)]  each connected component of $W_i$ corresponds to exactly one sequence  of $i+2$ critical manifolds (starting with one at the min level and ending with one at the max level)  which can be connected by trajectories.  
\end{itemize}
\end{theorem}
\begin{proof}(Sketch) The manifold with corners $W$ is obtained inductively by setting $W_0$ to be the boundary of the blow-up of the minimal critical manifolds. One has a natural smooth surjective map $\pi_0:[\min{f}, \min{f}+\epsilon]\times W_0\ra f^{-1}(\min{f}, \min{f}+\epsilon])$. Now $W_1$ is the blow-up of $W_0$ along $\pi_0^{-1}(\bigcup_F S(F)\cap f^{-1}(\min{f}+\epsilon))$ where the union is after all critical manifolds $F$  lying at the next critical level. The manifold $W_1$ is a manifold with boundary. The process goes one through induction.     
\end{proof}

\section{Residue computations in the unitary group}\label{resU}

Consider the following  submanifold of the unitary group $U(n)$:
 \[ \mathscr{U}_{\{n\}}:=\{U\in \mathscr{U}(n)~|~\dim{\Ker{(1-U)}}=n-1\}\subset U(n).\] 
 
 Let 
  $c_{n-1/2}=-\left(\frac{i}{2\pi}\right)^n\frac{[(n-1)!]^2}{(2n-1)!}\tr \wedge^{2n-1}g^{-1}dg\in \Omega^{2n-1}(U(n))$.
 The purpose of this section is to prove the following:
 \begin{prop} \[\int_{\mathscr{U}_{\{n\}}} c_{n-1/2}=1,\]
 \end{prop}
 In what follows $\bC\bP^{n-1}$ is the set of \emph{lines} passing through the origin in $\bC^n$. We consider the map:
  \[\phi: S^1\times \bC\bP^{n-1}\ra U(n),\quad\quad (\lambda,L)\ra \left(\begin{array}{cc} \lambda & 0\\ 0&  1\end{array}\right),
  \]
  where the block decomposition is relative $L\oplus L^{\perp}$. The restriction of $\phi$ to $S^1\setminus\{1\}\times \bC\bP^{n-1}$ is a diffeomorphism onto $\mathscr{U}_n$ and the image of $\phi$ is $\mathscr{U}_{\{n\}}\cup\{\id_{\bC^n}\}$.

\begin{remark}\label{orient disc} When integrating forms one has to fix an orientation. The orientation we chose is the one obtained  by \emph{declaring} the above map $\phi:S^1\times\bC\bP^{n-1}\ra U(n)$ (which is a diffeomorphism onto $U_{n}$ away from a point) to be orientation preserving. The manifold $S^1\times \bC\bP^{n-1}$ has a canonical orientation.  Another natural option would be to use the projective space of hyperplanes in $\bC^{n}$ and modify the map $\phi$ accordingly. If we insist that the integral is $1$ that would have the effect of changing\footnote{The first coefficient in $c_{n-1/2}$ would be "the nicer"  $\left(\frac{1}{2\pi i}\right)^n$ instead of $-\left(\frac{i}{2\pi}\right)^n$.} $c_{n-1/2}$ by  $(-1)^{n-1}$    since the map $\bC\bP^{n-1}\ra \bC\bP^{*,n-1}$   which takes $L\ra L^{\perp}$  becomes $A\ra A^*$ when linearized.  The  convention we chose fits the computations in \cite{Ni3}, Remark 6.5).
\end{remark}

   Fix $L\in\bC\bP^{n-1}$. Written in a chart of type $S^1\times \Hom(L,L^{\perp})$ the map $\phi$ has the following expression:
  \[  (\lambda, A)\ra \left(\begin{array}{cc} \frac{\lambda+A^*A}{1 +A^*A} & (1+A^*A)^{-1}A^*(\lambda-1)\\
  (1+AA^*)^{-1}A(\lambda-1) & \frac{1+\lambda AA^*}{1+AA^*}\end{array}\right)
  \]
  The differential at a point $(\lambda, 0)$  is:
  \[ d\phi_{\lambda,L}(w,S)=\left(\begin{array}{cc}w& S^*(\lambda-1)\\
   S(\lambda-1) &0 \end{array}\right).
  \]
  Then \[\phi^{-1}(\lambda,L)\cdot d\phi_{\lambda,L}(w,S)=\left(\begin{array}{cc}  \bar{\lambda} w & S^*(1-\bar{\lambda})\\
   S(\lambda-1) & 0
   \end{array}\right). \]
   This is of course the pull-back of $g^{-1}dg$ to $S^{1}\times \bC\bP^{n-1}$ as a form with values in $\mathfrak{u}(\tau\oplus\tau^{\perp})$ where $\tau$ represents the pull-back of the tautological bundle to $S^1\times \bC\bP^{n-1}$. We write it as:
   \[ \omega =\left(\begin{array}{cc} \lambda^{-1}d\lambda & -\overline{\alpha(\lambda)} dS^*\\
    {\alpha(\lambda)}dS & 0.
   \end{array}\right).\quad\quad\quad \alpha(\lambda)=\lambda-1
   \]
  where $dS$ is the pull-back to $S^1\times\bC\bP^{n-1}$ of the $1$-form with values in the bundle $T\bC\bP^{n-1}=T^{(1,0)}\bC\bP^{n-1}=\Hom(\tau,\tau^{\perp})$ obtained by differentiating the identity on $\bC\bP^{n-1}$. Moreover $dS^*$ is the conjugate of $dS$ which can be seen as a form with values in the dual to $T^{(1,0)}\bC\bP^{n-1}$ via metric duality.
  
  We split $\omega=C+ B$ where $C=\left(\begin{array}{cc}  \lambda^{-1}d\lambda & 0\\
  0 &0.
   \end{array}\right)$ and $B=\left(\begin{array}{cc} 0 & - \overline{\alpha(\lambda)} dS^*\\
   {\alpha(\lambda)}dS & 0
   \end{array}\right).$
   
   We will also need $C_1:=\left(\begin{array}{cc} 0 &0 \\
  0 & - \lambda^{-1}d\lambda \otimes\id
   \end{array}\right).$
  The following relations are straightforward:
  \[C^2=0,\qquad C_1^2=0,\qquad B^2C=CB^2,\qquad B^2C_1=C_1B^2
  \]
  and a quick computation shows that 
  \[BCB=B^2C_1,\qquad C_1BC=0,\qquad CBC_1=0.
  \]
   \begin{lemma} The following equality holds where we set $\wedge^0 B=\id$:
 \[ \wedge^{2n-1}\omega= \wedge^{2n-2}B\wedge [n C+(n-1)C_1] +\wedge^{2n-1}B,\quad\quad\forall n\geq 1.
 \]
 \end{lemma}
 \begin{proof} We prove this by induction. For $n=1$ this is clearly true. Notice that $\wedge^2\omega= CB+BC +B^2$. It is quite easy to see that the monomials that contain $C$ twice or one $C$ and one $C_1$ vanish. Instead of writing $\wedge^{2n-1}\omega$, we write $\omega^{2n-1}$. We have due to the relations above
 \[\omega^{2n-1}=\omega^{2n-3}\wedge \omega^2=[B^{2n-4}((n-1)C+(n-2)C_1)+B^{2n-3}]
(CB+BC +B^2)=
 \]
 \[ =B^{2n-4}((n-1)C+(n-2)C_1)B^2+B^{2n-3} CB+B^{2n-2}C+B^{2n-1}=
 \]
 \[B^{2n-2}((n-1)C+(n-2)C_1)+B^{2n-2}C_1+B^{2n-2}C+B^{2n-1}=B^{2n-2}(nC+(n-1)C_1)+B^{2n-1}.
 \]
 \end{proof}
Since $B^{2n-1}$ is block anti diagonal we have due to the lemma:
 \[\tr {\wedge^{2n-1}\omega}=n\tr (B^{2n-2}\wedge C)+(n-1)\tr (B^{2n-2}\wedge C_1).
 \]
  We write this as 
  \begin{equation}\label{comp2} \tr{\wedge^{2n-1}\omega}=(-1)^{n-1}|\alpha(\lambda)|^{2n-2}(\lambda^{-1}d\lambda)\cdot \wstr{(D^{n-1})},\end{equation} where
   \[D=\left(\begin{array}{cc} dS^*\wedge dS&0 \\
   0& dS\wedge dS^*   \end{array}\right)\]and\[ \wstr{T}=n\tr{T_2}-(n-1)\tr{T_1}\qquad \forall T=\left(\begin{array}{cc} T_1 &0 \\ 0 & T_2\end{array}\right).\]

We use  the Cayley transform that preserves the orientation $t\ra \frac{t-i}{t+i}$ to turn the integral
 \[ \int_{S^1}(-1)^{n-1}|\alpha(\lambda)|^{2n-2}\lambda^{-1}d\lambda=(-1)^{n-1}2^{n-1}\int_{S^1}(1-\Real \lambda)^{n-1}\lambda^{-1}d\lambda.
 \]
into  the integral
 \[ 2^{2n-1}(-1)^{n-1}i \int_{\bR}\frac{1}{(1+t^2)^n}~dt
 \]
which we compute  with the help of the Residue Theorem,  therefore obtaining
 \begin{equation} \label{comp3} \int_{S^1}(-1)^{n-1}|\alpha(\lambda)|^{2n-2}\lambda^{-1}d\lambda=2^{2n-1}(-1)^{n-1}i\int_{\bR}\frac{1}{(1+t^2)^n}~dt= (-1)^{n-1}2\pi i{2n-2\choose n-1}.
 \end{equation}
 
In order to compute $\displaystyle\int_{\bC\bP^{n-1}}\wstr{(\wedge^{n-1} D)}$ we remark that  by its very definition $D$ satisfies:
\[ D_{UL}=\ad_{U}D_L\quad\quad\forall U\in U(n),
\]
which implies that $\wstr{(\wedge^{n-1} D)}$ is actually an $U(n)$ invariant form of maximum degree in $\bC\bP^{n-1}$. Therefore it has to be a constant multiple of the volume form on $\bC\bP^{n-1}$. To determine this multiple fix a point $L_0=[1:0:\ldots :0]\in \bC\bP^{n-1}$ and denote  by $dz_1,d\bar{z_1},\ldots, dz_{n-1}, d\bar{z}_{n-1}$ the canonical $1$-forms which form a basis in the dual space to $T^{(1,0)}_{L_0}\bC\bP^{n-1}\oplus T^{(0,1)}_{L_0}\bC\bP^{n-1}$. Then
\[ dS_{L_0}=\left(\begin{array}{c} dz_1 \\ dz_2\\ \ldots \\ dz_{n-1}\end{array}\right)\quad\quad dS^*_{L_0}=\left(d\overline{z}_1, d\overline{z}_2,\ldots, d\overline{z}_{n-1} \right)
\]
Now $dS^*_0\wedge dS_0=\sum_{i=1}^nd\overline{z_i}\wedge dz_i$ and hence the top-left entry of of $\wedge^{n-1}D$ is
\begin{equation}\label{comp5} (n-1)!(-1)^{n-1} dz_1\wedge d\overline{z_1}\wedge\ldots \wedge dz_{n-1}\wedge d\overline{z}_{n-1}.
\end{equation} 
On the other hand for the matrix of two forms $(dS_{L_0}\wedge dS_{L_0}^*)_{ij}=dz_i\wedge d\bar{z}_j$ we have the following:
\begin{lemma}\label{comp6}\[(dS_{L_0}\wedge dS_{L_0}^*)^{n-1}=(n-2)!(-1)^{n-2}dz_1\wedge d\overline{z}_1\wedge\ldots \wedge dz_{n-1}\wedge d\overline{z}_{n-1}\otimes \id.\]
\end{lemma}
\begin{proof} The $ij$ entry of $(dS_{L_0}\wedge dS_{L_0}^*)^{n-1}$ is a huge sum:
\[ \sum_{k_1,\ldots, k_{n-2}} a_{ik_1}a_{k_1k_2}\ldots a_{k_{n-2}j},
\]
where $a_{lp}$ is an entry of $dS_{L_0}\wedge dS_{L_0}^*$. It is then straightforward that $(dS_{L_0}\wedge dS_{L_0}^*)^{n-1}$ is diagonal and every diagonal entry is up to a sign $(n-2)!dz_1\wedge d\bar{z}_{1}\wedge\ldots\wedge  dz_{n-1}\wedge d\bar{z}_{n-1}$. The sign requires a bit of care.
\end{proof}
We therefore get from Lemma \ref{comp6} and (\ref{comp5}) that
  \[ \wstr{D^{n-1}_{L_0}}=(-1)^{n-1}(2n-1) (n-1)!dz_1\wedge d\overline{z}_1\wedge\ldots \wedge dz_{n-1}\wedge d\overline{z}_{n-1}.
 \]
Recall now that the standard Kahler form on $\bC\bP^{n-1}$  at the point $L_0$ (see \cite{GH}, page 31) is:
\[ \eta:=\frac{i}{2\pi}\sum_{j=1}^{n-1}dz_j\wedge d\bar{z}_j
\]
and 
\[ \int_{\bC\bP^{n-1}}\eta^{\wedge n-1}=1.
\]
Due to the fact that we are dealing with invariant forms we get that the identity
 \[\wstr{\wedge^{n-1}D}=(-1)^{n-1}(2n-1)\left(\frac{2\pi}{i}\right)^{n-1}\eta^{\wedge n-1}\]  holds everywhere and therefore:
\begin{equation}\label{comp4}\int_{\bC\bP^{n-1}} \wstr{\wedge^{n-1}D}=(-1)^{n-1}(2n-1)\left(\frac{2\pi}{i}\right)^{n-1}.
\end{equation}
Putting together Fubini Theorem together with (\ref{comp2}), (\ref{comp3}) and (\ref{comp4}) finishes the proof of the proposition.
\end{appendix}

\end{document}